\documentclass[11pt,english]{scrartcl}
\usepackage[T1]{fontenc}
\usepackage[latin9]{inputenc}
\usepackage[a4paper]{geometry}
\usepackage{babel}
\usepackage{refstyle}
\usepackage{wrapfig}
\usepackage{amsthm}
\usepackage{amsmath}
\usepackage{amssymb}
\usepackage{esint}
\usepackage[authoryear]{natbib}
\usepackage{lmodern}
\usepackage{graphics}
\usepackage[kerning=true]{microtype}
\usepackage[unicode=true,pdfusetitle,
 bookmarks=true,bookmarksnumbered=false,bookmarksopen=false,
 breaklinks=false,pdfborder={0 0 0},backref=false,colorlinks=true]
 {hyperref}

\makeatletter

\theoremstyle{plain}
\newtheorem{thm}{Theorem}
\theoremstyle{definition}
\newtheorem{defn}[thm]{Definition}
\theoremstyle{remark}
\newtheorem{rem}[thm]{Remark}
\theoremstyle{plain}
\newtheorem{lem}[thm]{Lemma}
\theoremstyle{plain}
\newtheorem{prop}[thm]{Proposition}

\newcommand{\includefigure}[1]{\centering\includegraphics{figures/#1}}

\setcapindent{0pt}

\global\long\def\i{\mathrm{i}}
\global\long\def\e{\mathrm{e}}
\global\long\def\rd{\mathrm{d}}
\global\long\def\sign{\operatorname{sign}}
\global\long\def\bu{\boldsymbol{u}}
\global\long\def\bv{\boldsymbol{v}}
\global\long\def\bw{\boldsymbol{w}}
\global\long\def\bR{\boldsymbol{R}}
\global\long\def\bd{\boldsymbol{d}}
\global\long\def\br{\boldsymbol{r}}
\global\long\def\bq{\boldsymbol{q}}

\global\long\def\bx{\boldsymbol{x}}
\global\long\def\by{\boldsymbol{y}}

\global\long\def\be{\boldsymbol{e}}

\global\long\def\bR{\boldsymbol{R}}

\global\long\def\bn{\boldsymbol{n}}
\global\long\def\ba{\boldsymbol{a}}

\global\long\def\bphi{\boldsymbol{\varphi}}
\global\long\def\bnabla{\boldsymbol{\nabla}}
\global\long\def\bcdot{\boldsymbol{\cdot}}

\global\long\def\bzero{\boldsymbol{0}}
\global\long\def\JH{\textsc{jh}}

\begin{document}

\title{On the stationary Navier-Stokes\\
equations in the half-plane}

\author{\href{mailto:julien.guillod@unige.ch}{Julien Guillod} and \href{mailto:peter.wittwer@unige.ch}{Peter Wittwer}\\
{\small{Department of Theoretical Physics,}}\\
{\small{University of Geneva, Switzerland}}}

\maketitle
\begin{abstract}
We consider the stationary incompressible Navier-Stokes equation in
the half-plane with inhomogeneous boundary condition. We prove existence
of strong solutions for boundary data close to any Jeffery-Hamel solution
with small flux evaluated on the boundary. The perturbation of the
Jeffery-Hamel solution on the boundary has to satisfy a nonlinear
compatibility condition which corresponds to the integral of the velocity
field on the boundary. The first component of this integral is the flux
which is an invariant quantity, but the second, called the asymmetry,
is not invariant, which leads to one compatibility condition. Finally,
we prove existence of weak solutions, as well as weak-strong uniqueness
for small data.
\end{abstract}
\textit{\small{MSC class:}}{\small{ 76D03, 76D05, 35Q30, 76D25, 74F10,
76M10}}{\small \par}

\section{Introduction\label{sec:introduction}}

The stationary and incompressible Navier-Stokes equations in the half-plane\index{Half-plane}
\[
\Omega=\left\{ (x,y)\in\mathbb{R}^{2}\colon y>1\right\} 
\]
are
\begin{equation}
\begin{aligned}\Delta\bu-\bnabla p & =\bu\bcdot\bnabla\bu\,, & \bnabla\bcdot\bu & =0\,,\\
\left.\bu\right|_{\partial\Omega} & =\bu^{*}\,, & \lim_{|\bx|\to\infty}\bu & =\bzero\,,
\end{aligned}
\label{eq:ns}
\end{equation}
where $\bu^{*}$ is a boundary condition. Due to the incompressibility
of the fluid, the flux is an invariant quantity,
\[
\Phi=\int_{\partial\Omega}\bu^{*}\bcdot\bn=\int_{\mathbb{R}}v(x,y)\,\rd x\,,
\]
for all $y\geq1$, where $\bu=(u,v)$ and $\bn=(0,1)$ is the normal
vector to the half-plane. This problem (see \figref{scheme}a)
presents three difficulties: $\Omega$ is a two-dimensional unbounded
domain, the boundary of $\partial\Omega$ is unbounded, and the boundary
data are not zero. There is not much previous work on this problem,
but some authors have treated related problems. Concerning the half-plane
problem, \citet[\S 5]{Heywood-uniquenessquestions1976} proves the
uniqueness of solutions for the steady Stokes equation and the time-dependent
Navier-Stokes equation. The so called Leray's problem, which consists
of a finite number of outlets connected to a compact domain, has been
studied in detail by \citet{Amick-SteadysolutionsNavier1977,Amick-PropertiessteadyNavier-Stokes1978,Amick-SteadysolutionsNavier1980}
and several other authors, but the resolvability for large fluxes
is still an open problem. \citet{Fraenkel-LaminarFlow1962,Fraenkel-LaminarFlow1962b}
provides a formal asymptotic expansion of the stream function in case
of a curved channel by starting with the Jeffery-Hamel solution \citep{Jeffery-two-dimensionalsteadymotion1915,Hamel-SpiralfoermigeBewegungen1917}
for the first order. The case of paraboloidal outlets was first treated
by \citet{Nazarov-Asymptoticssolutions1998}, and then more recently
by \citet{Kaulakyte-NonhomogeneousBoundaryValue,Kaulakyte-NonhomogeneousBoundaryValue2013}.
Another important class of similar problems are the aperture domains,
introduced by \citet{Heywood-uniquenessquestions1976}, as shown in
\figref{scheme}b. The linear approximation was studied
in any dimension by \citet{Farwig-Notefluxcondition1996,Farwig-Helmholtzdecomposition1996}.
The three-dimensional case was treated by \citet{Borchers-Existenceuniqueness1992},
as well as other authors. For the two-dimensional nonlinear problem,
\citet{Galdi-Existenceandasymptotic1995} proved that the velocity
tends to zero in the $L^{2}$-norm for arbitrary values of the flux.
For small fluxes, \citet{Galdi-Existenceuniquenessand1996} and \citet{Nazarov-two-dimensionalapertureproblem1996}
show that the asymptotic behavior is given by a Jeffery-Hamel solution
but only if the problem is symmetry with respect to the $y$-axis.
The asymptotic behavior of the two-dimensional aperture problem in
the nonsymmetric case is still open. Finally, \citet{Nazarov.etal-Steadyflowsof2001,Nazarov.etal-Steadyflowsof2002}
considered a straight channel connected to a half-plane (see \figref{scheme}c),
and looked under which conditions the asymptotic behavior is given
by a Jeffery-Hamel flow in the half-plane and by the Poiseuille flow
in the channel. Theses conditions are described in details later on.
On a more applied side, the bifurcation properties and the stability
of the Jeffery-Hamel flows have retained the attention of many authors
\citep{Moffatt-Localsimilaritysolutions1980,Sobey-Bifurcationsoftwo-dimensional1986,Banks.etal-perturbationsofJeffery-Hamel1988,Uribe-stabilityJefferyHamel1997,Drazin-Flowthroughdiverging1999}.
\begin{figure}
\includefigure{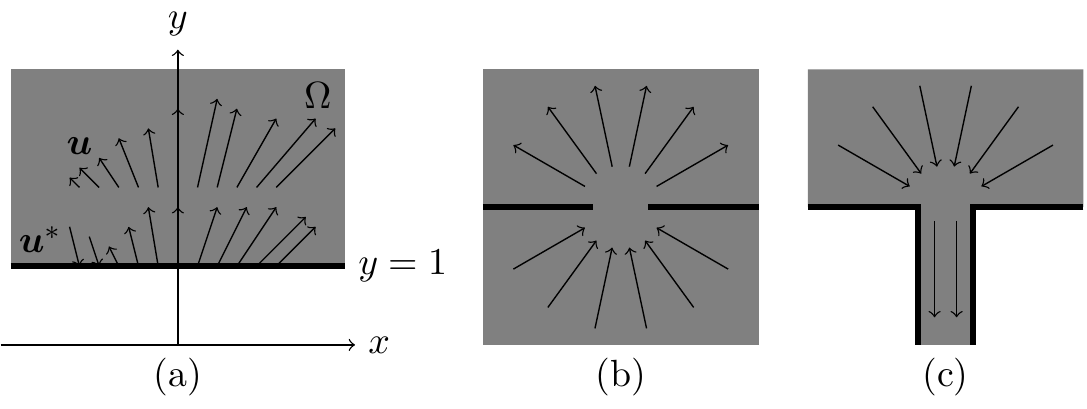}

\caption{\label{fig:scheme} (a) The domain $\Omega$ we consider
is the half plane defined by $x\in\mathbb{R}$ and $y>1$; (b) an
aperture domain; (c) a channel connected to a half-plane.}
\end{figure}

Jeffery-Hamel flows play an important role in the asymptotic behavior
of flows carrying flux. They own their name to the work of \citet{Jeffery-two-dimensionalsteadymotion1915,Hamel-SpiralfoermigeBewegungen1917},
and are radial scale invariant solutions of the two-dimensional stationary
incompressible Navier-Stokes equations
\begin{align*}
\Delta\bu-\left(\bu\bcdot\bnabla\right)\bu-\bnabla p & =\bzero\,, & \bnabla\bcdot\bu & =0\,,
\end{align*}
in domains
\[
D=\left\{ \left(r\sin(\theta),r\cos(\theta)\right)\in\mathbb{R}^{2}\colon\, r>0\text{ and }\theta\in\left(-\beta,\beta\right)\right\} 
\]
with $\beta\in\bigl(0,\tfrac{\pi}{2}\bigr]$, satisfying the boundary
condition
\[
\bu\bigr|_{\partial D\setminus\left\{ \bzero\right\} }=\bzero\,.
\]
Explicitly, a Jeffery-Hamel\index{Jeffery-Hamel solutions}\index{Exact solutions!Jeffery-Hamel}
solution $\bu_{\JH}$ is of the form
\[
\bu_{\JH}(r,\theta)=\dfrac{1}{r}f(\theta)\,\be_{r}\,,
\]
with $f$ a solution of the nonlinear second order ordinary differential
equation
\[
f^{\prime\prime}+f^{2}+4f=2C\,,
\]
with $C\in\mathbb{R}$, satisfying the boundary condition $f(\pm\beta)=0$.
The constant $C$ is related to the flux $\Phi$ of the flow,
\[
\Phi=\int_{-\beta}^{+\beta}f(\theta)\,\rd\theta\,.
\]

The Jeffery-Hamel solutions have been intensively studied \citep{Rosenhead-SteadyTwo-DimensionalRadial1940,Fraenkel-LaminarFlow1962,Banks.etal-perturbationsofJeffery-Hamel1988,Tutty-Nonlineardevelopmentof1996},
and, because some mathematical questions still remain open, there
has been a regain of interest in recent years \citep{Rivkind.Solonnikov-JefferyHamelAsymptoticsSteady2000,Kerswell.etal-Steadynonlinearwaves2004,Putkaradze.Vorobieff-InstabilitiesBifurcations2006,Corless.Assefa-Jeffery-hamelflowwith2007}.

In what follows we are interested in the half-plane case, so we consider
$\beta=\tfrac{\pi}{2}$, \emph{i.e.}, when the domain is the upper
half plane $D=\mathbb{R}\times\left(0,\infty\right)$. In Cartesian
coordinates, the two components of the velocity of the Jeffery-Hamel
solutions are
\begin{align}
u_{\JH}(x,y) & =\frac{1}{y}f_{u}\left(\frac{x}{y}\right)\,, & v_{\JH}(x,y)= & \frac{1}{y}f_{v}\left(\frac{x}{y}\right)\,,\label{eq:JH-xy}
\end{align}
with
\begin{align*}
f_{u}(s) & =\frac{s\, f(\arctan s)}{1+s^{2}}\,, & f_{v}(s) & =\frac{f(\arctan s)}{1+s^{2}}\,.
\end{align*}
\begin{wrapfigure}[16]{r}{55mm}%
\vspace{-8pt}\includefigure{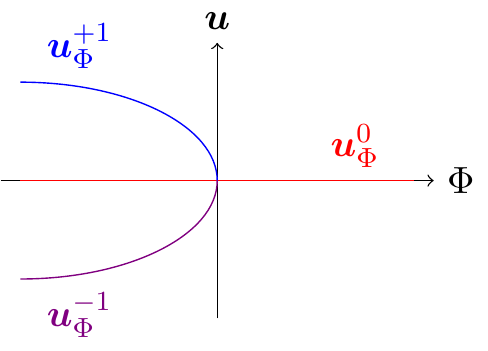}

\caption{\label{fig:bifurcation-JH}Existence of the Jeffery-Hamel
flows for small values of the flux $\Phi$. For $\Phi>0$, there exists
one symmetric solutions, but for $\Phi<0$, also two additional asymmetric
solutions exist.}
\end{wrapfigure}%
The Jeffery-Hamel solutions for $\beta=\tfrac{\pi}{2}$ have a peculiar
property: for small $\Phi<0$ there is more than one solution. In
fact, as shown in appendix~\S\ref{sec:JH-phi}, $\Phi=0$ is a tri-critical
bifurcation point, see \figref{bifurcation-JH}. For small
$\Phi>0$ the Jeffery-Hamel problem has a solution $\bu_{\Phi}^{0}$
which is symmetric with respect to the $y$-axis. The solution $\bu_{\Phi}^{0}$
also exists for small values of $\Phi<0$, but when crossing from
$\Phi>0$ to $\Phi<0$, an additional pair $\bu_{\Phi}^{\pm1}$ of
asymmetric solutions (related to each other by a reflection with respect
to the $y$-axis) appears. For $\Phi=0$, the Jeffery-Hamel solution
is the zero function, and will be ignored in what follows.

The central idea of the method we use to study \eqref{ns},
is to interpret the system as an evolution equation with $y$ playing
the role of time. The boundary data of the original problem then become
the initial data for the resulting Cauchy problem. This allows discussing
the ``time\textquotedblright{} dependence of quantities like the
flux $\Phi=\int_{\mathbb{R}}v(x,y)\mathrm{d}x$ and the asymmetry
$A=\int_{\mathbb{R}}u(x,y)\mathrm{d}x$ in a natural setting. We assume
for the moment sufficient decay for these integrals to make sense.
As can be seen from \eqref{JH-xy}, the flux and the asymmetry
are invariants for a Jeffery-Hamel solution $\bu_{\Phi}^{\sigma}$,\textit{
i.e.}, they are independent of the time $y$, and therefore\index{Flux}\index{Asymmetry}
\begin{align*}
A & =\int_{\mathbb{R}}u_{\Phi}^{\sigma}(x,1)\,\rd x=\int_{\mathbb{R}}u_{\Phi}^{\sigma}(x,y)\,\rd x=\lim_{y\rightarrow\infty}\int_{\mathbb{R}}u_{\Phi}^{\sigma}(x,y)\,\rd x\,,\\
\Phi & =\int_{\mathbb{R}}v_{\Phi}^{\sigma}(x,1)\,\rd x=\int_{\mathbb{R}}v_{\Phi}^{\sigma}(x,y)\,\rd x=\lim_{y\rightarrow\infty}\int_{\mathbb{R}}v_{\Phi}^{\sigma}(x,y)\,\rd x\,.
\end{align*}
The flux is an invariant of the Navier-Stokes equation \eqref{ns},
\emph{i.e.}, if $\bu$ is a solution of the Navier-Stokes equation,
then for $y>1$,

\[
\int_{\mathbb{R}}v(x,1)\,\rd x=\int_{\mathbb{R}}v(x,y)\,\rd x=\lim_{y\rightarrow\infty}\int_{\mathbb{R}}v(x,y)\,\rd x\,,
\]
but the asymmetry is not an invariant, so typically, 
\[
\int_{\mathbb{R}}u(x,1)\,\rd x\neq\int_{\mathbb{R}}u(x,y)\,\rd x\neq\lim_{y\rightarrow\infty}\int_{\mathbb{R}}u(x,y)\,\rd x\,.
\]
As we will see below, for solutions that are close but not equal to
Jeffery-Hamel, the asymmetry is no more an invariant, and this fact
is the main source of trouble for the construction of solutions.

Jeffery-Hamel solutions are singular at the origin, and, in order
to study this phenomenon, \emph{i.e.} look for solutions which are
close to Jeffery-Hamel flows, it is necessary to regularize the problem.
For this purpose, given a Jeffery-Hamel solution $\bu_{\Phi}^{\sigma}$
in the upper half plane $D=\left\{ (x,y)\in\mathbb{R}^{2}\colon y>0\right\} $,
we restrict it to the domain $\Omega=\left\{ (x,y)\in\mathbb{R}^{2}\colon y>1\right\} $,
and construct stationary solutions of the Navier-Stokes equations
which are close to the Jeffery-Hamel flow by imposing boundary conditions
of the form 
\begin{equation}
\bu\bigr|_{\partial\Omega}=\bu_{\Phi}^{\sigma}\bigr|_{\partial\Omega}+\bu_{b}\,,\label{eq:BC-ub}
\end{equation}
with $\bu_{b}$ small and with zero flux,
\[
\int_{\partial\Omega}\bu_{b}\bcdot\bn=0\,.
\]
Even when considering such boundary conditions, we are not able to
perform a fixed point argument on the nonlinearity by inverting the
Stokes problem. The main reason is that the flux is determined by
the boundary condition, while the asymmetry is not. In order to adjust
the asymmetry, we rewrite the boundary condition as
\begin{equation}
\left.\bu\right\vert _{\partial\Omega}=\left.\bu_{\Phi}^{\sigma}\right\vert _{\partial\Omega}+\left(A,0\right)\frac{1}{\sqrt{2\pi}}\e^{-\frac{1}{2}x^{2}}+\bu_{s}\,,\label{eq:BC-ur}
\end{equation}
where
\[
A=\int_{\mathbb{R}}u_{b}\,,
\]
so that $\bu_{s}$ has no asymmetry and no flux.
\[
\int_{\mathbb{R}}\bu_{s}=\bzero\,.
\]
The choice of $\e^{-\frac{1}{2}x^{2}}$ in \eqref{BC-ur}
is for convenience later on, and we could have chosen instead any
other smooth function of rapid decay.

We will show the existence of strong solutions to the Navier-Stokes
equation \eqref{ns}, with the boundary condition \eqref{BC-ur}
for $\Phi$ small and $\bu_{s}$ in a small ball by adjusting the
parameter $A$. The main result is the following:
\begin{thm}
For all boundary condition of the form \eqref{BC-ur} with
$\bu_{\Phi}^{\sigma}$ a Jeffery-Hamel solution with small enough
flux $\Phi$ and all $\bu_{s}$ in a small enough neighborhood of
zero in some function space, there exists a solution $(\bu,p)$ of
the Navier-Stokes equation in $\Omega$, satisfying\index{Navier-Stokes equations!in the half-plane!asymptotic behavior}\index{Asymptotic behavior!in the half-plane!Navier-Stokes solutions!strong
solutions}
\[
\lim_{y\rightarrow\infty}y\left(\sup_{x\in\mathbb{R}}\bigl|\bu-\bu_{\Phi}^{\sigma}\bigr|\right)=0\,,
\]
and
\begin{align}
\begin{aligned}\bnabla\bu & \in L^{2}(\Omega)\,,\\
\bu/y & \in L^{2}(\Omega)\,,
\end{aligned}
 &  & \begin{aligned}y\bu & \in L^{\infty}(\Omega)\,,\\
y^{2}\bnabla\bu & \in L^{\infty}(\Omega)\,.
\end{aligned}
\label{eq:bounds-strong}
\end{align}
Moreover, if $\bv$ is a weak solution (defined in \eqref{weak-sol})
of \eqref{ns} and $\bu$ a strong solution of \eqref{ns}
satisfying \eqref{bounds-strong}, then $\bu=\bv$, provided
$\bu^{*}$ is small enough.
\end{thm}
We now discuss more precisely the results of \citet{Nazarov.etal-Steadyflowsof2001,Nazarov.etal-Steadyflowsof2002},
for the domain shown in \figref{scheme}c. We note that
in this domain, the flux through the channel is not prescribed. They
show that by requiring the asymptotic behavior to be an antisymmetric
Jeffery-Hamel solution $\bu_{\Phi}^{\pm1}$, there exists a unique
solution in some weighted space, and the flux is uniquely determined
by the data. Conversely, by requiring that the asymptotic behavior
is given by a symmetric Jeffery-Hamel solution $\bu_{\Phi}^{0}$,
the Navier-Stokes equation linearized around $\bu_{\Phi}^{0}$ leads
to a well-posed problem for $\Phi<0$ and to an ill-posed one for
$\Phi>0$. So for $\Phi<0$, there exists a unique solution for all
small enough fluxes, but for $\Phi>0$, the asymptotic behavior is
still unknown. So we believe that in case $\Phi<0$, the Navier-Stokes
equation in the half-plane \eqref{ns} has a solution decaying
like $r^{-1}$ at infinity whose asymptote is given by a Jeffery-Hamel
solution, but in the case $\Phi>0$, it is still not clear that the
solution is in general bounded by $r^{-1}$.

The remainder of this paper is organized as follows. In \secref{halplane-spaces},
we introduce the function spaces which we use for the mathematical
formulation of the problem, and prove some basic bounds. In \secref{stokes},
we rewrite the Stokes equation as a dynamical system, present the
associated integral equations, and provide bounds on the solution
of the Stokes system, so that in \secref{ns}, we can show
the existence of strong solutions to the Navier-Stokes system. In
\secref{weaksol}, we prove existence of weak solutions,
and, finally, in \secref{uniqueness}, we prove uniqueness
of solutions for small data with a weak-strong uniqueness result.
In the last part, we also present numerical simulations that show
that the asymptotic behavior is most likely not given by the Jeffery-Hamel
solution if $\Phi>0$. In the appendix, we show the existence of symmetric
and asymmetric Jeffery-Hamel solutions with small flux.

\section{Function spaces\label{sec:halplane-spaces}}

As explained in the introduction, our strategy of proof is to rewrite
\eqref{ns} as a dynamical system with $y$ playing the
role of time. This system is studied by taking the Fourier transform
in the variable $x$, which transforms the system into a set of ordinary
differential equations with respect to $y$. We now define the function
spaces for the Fourier transforms of the velocity field, pressure
field and the nonlinearity. The choice of spaces is motivated by the
scaling property of the equations with respect to $x$ and $y$ when
linearized around a Jeffery-Hamel solution. This setup turns out to
be natural for the description of the asymptotic behavior of solutions
close to Jeffery-Hamel flows. Similar function spaces were already
used by \citet{Wittwer-structureofStationary2002}, where the basic
operations which are needed for the discussion of the Navier-Stokes
equations were discussed. In particular, \citet{Wittwer-structureofStationary2002}
shows basic bounds on the convolution with respect to the variable
$k$, the Fourier conjugate variable of $x$, which is needed to implement
the nonlinearities, and bounds on the convolution with the semigroup
$\e^{-\left|k\right|y}$ which is associated with the Stokes operator
when viewed as a time evolution in $y$. Further properties and improved
bounds have been proved by \citet{Hillairet.Wittwer-Existenceofstationary2009,Boeckle.Wittwer-Decayestimatessolutions2012}.
\begin{defn}[Fourier transform and convolution]
\label{def:fourier-convolution}For two functions $\hat{f}$
and $\hat{g}$ defined almost everywhere in $\Omega$ and which are
in $L^{1}(\mathbb{R})$ for all $y\geq1$, the inverse Fourier transform
of $\hat{f}$ is defined by
\[
f(x,y)=\mathcal{F}\bigl[\hat{f}\,\bigr](x,y)=\int_{\mathbb{R}}\e^{\i kx}\hat{f}(k,y)\,\rd k\,,
\]
and the convolution by
\[
\bigl(\hat{f}*\hat{g}\bigr)(k,y)=\int_{\mathbb{R}}\hat{f}(k-\ell,y)\hat{g}(\ell,y)\,\rd\ell\,.
\]

\end{defn}
We note that with these definitions,
\[
fg=\mathcal{F}\bigl[\hat{f}*\hat{g}\bigr]\,.
\]

We now define two families of function spaces: the first one is for
functions of $k$ only which will be used for the boundary data, and
the second one is for functions of $k$ and $y$:
\begin{defn}[function spaces on $\partial\Omega$]
\label{def:function-spaces-boundary}For $\alpha\geq0$
and $q\in\mathbb{R}$, let $\mathcal{A}_{\alpha,q}$ be the Banach
space of functions $\hat{f}\in C(X,\mathbb{C})$ where
\begin{equation}
X=\begin{cases}
\mathbb{R}\,, & q\geq0\,,\\
\mathbb{R}\setminus\{0\}\,, & q<0\,,
\end{cases}\label{eq:defX}
\end{equation}
such that $\overline{\hat{f}(k)}=\hat{f}(-k)$ and such that the norm\index{Space!A@$\mathcal{A}_{\alpha,q}$}
\[
\bigl\Vert\hat{f};\mathcal{A}_{\alpha,q}\bigr\Vert=\sup_{X}\frac{\bigl|\hat{f}\bigr|}{\eta_{\alpha,q}}\qquad\text{with}\qquad\eta_{\alpha,q}(k)=\frac{\left|k\right|^{q}}{1+\left|k\right|^{\alpha+q}}\,,
\]
is finite. For $\alpha\geq0$ and $q\geq0$, let $\mathcal{T}_{\alpha,q}$
and $\mathcal{W}_{\alpha,q}$, be the Banach space of functions in
$\mathcal{A}_{\alpha+1,\min(0,q-1)}$ such that their respective norm\index{Space!T@$\mathcal{T}_{\alpha,q}$}\index{Space!WW@$\mathcal{W}_{\alpha,q}$}
\begin{align*}
\bigl\Vert\hat{f};\mathcal{T}_{\alpha,q}\bigr\Vert & =\sum_{i=0}^{\lfloor q\rfloor}\bigl\Vert\partial_{k}^{i}\hat{f};\mathcal{A}_{\alpha+1,\min(0,q-1-i)}\bigr\Vert\,, & \bigl\Vert\hat{f};\mathcal{W}_{\alpha,q}\bigr\Vert & =\sum_{i=0}^{\lfloor q\rfloor}\bigl\Vert\partial_{k}^{i}\hat{f};\mathcal{A}_{\alpha+1,q-1-i}\bigr\Vert\,,
\end{align*}
is finite, where $\lfloor q\rfloor$ denotes the integer part of $q$.\end{defn}
\begin{rem}
The parameter $\alpha$ captures the behavior of functions at infinity,
which corresponds to the regularity in $x$ in direct space. For example,
if $\hat{f}\in\mathcal{A}_{\alpha,0}$ for $\alpha>1$, then $\hat{f}\in L^{1}(\mathbb{R})$
and by the dominated convergence theorem, $f\in C(\mathbb{R})$. The
index $q$ characterizes the behavior near $k=0$: a function which
behaves like $\left|k\right|^{q}$ around zero is in the space $\mathcal{A}_{\alpha,q}$.
The space $\mathcal{T}_{\alpha,q}$ includes some characterization
of the derivative with respect to $k$, which is needed in order to
characterized the behavior near $k=0$ as shown in the next lemma.\end{rem}
\begin{lem}
\label{lem:spaces-TtoU}For all $\hat{f}\in\mathcal{T}_{\alpha,q}$,
the function
\[
\hat{g}(k)=\hat{f}(k)-\left(\sum_{i=0}^{\lfloor q\rfloor-1}\frac{k^{i}}{i!}\partial_{k}^{i}\hat{f}(0)\right)\chi(\left|k\right|)\,,
\]
where $\chi$ is a smooth cut-off function with
\begin{align*}
\chi\left(\left[0,1\right]\right) & =\left\{ 1\right\} \,, & \chi\left(\left[2,\infty\right)\right) & =\left\{ 0\right\} \,,
\end{align*}
satisfies $\hat{g}\in\mathcal{W}_{\alpha,p}$ and therefore
\[
\mathcal{W}_{\alpha,q}=\left\{ \hat{f}\in\mathcal{T}_{\alpha,q}\colon\hat{f}^{(i)}(0)=0,\,\forall i\leq\lfloor q\rfloor-1\right\} \,.
\]
\end{lem}
\begin{proof}
Due to the fact that the behavior at large $\left|k\right|$ of functions
in $\mathcal{T}_{\alpha,q}$ and in $\mathcal{W}_{\alpha,q}$ are
the same we only need to prove the behavior for small $\left|k\right|$.
In view of the properties of the cut-off function, for $\left|k\right|\leq1$
and $i\leq\lfloor q\rfloor-1$, we have
\begin{align*}
\left|\partial_{k}^{i}\hat{g}(k)\right| & =\left|\partial_{k}^{i}\hat{f}(k)-\left(\sum_{j=0}^{\lfloor q\rfloor-1-i}\frac{k^{j}}{j!}\partial_{k}^{j+i}\hat{f}(0)\right)\chi(\left|k\right|)\right|\\
 & \leq\int_{0}^{\left|k\right|}\left|\partial_{k}^{\lfloor q\rfloor}\hat{f}(\xi)\right|\xi^{\lfloor q\rfloor-1-i}\rd\xi\lesssim\bigl\Vert\hat{f};\mathcal{T}_{\alpha,q}\bigr\Vert\left|k\right|^{q-i}\,.
\end{align*}

\end{proof}
For functions of $k$ and $y$, we define the following spaces with
norms reflecting the scaling property of the Jeffery-Hamel solution:
\begin{defn}[function spaces on $\Omega$]
\label{def:function-spaces-bulk}For $\alpha\geq0$ and
$q\in\mathbb{R}$, let $\mathcal{B}_{\alpha,q}$ be the Banach space
of functions $\hat{f}\in C(X\times\left[1,\infty\right),\mathbb{C})$
where $X$ is defined by \eqref{defX}, such that $\overline{\hat{f}(k,y)}=\hat{f}(-k,y)$
and such that the norm\index{Space!B@$\mathcal{B}_{\alpha,q}$}
\[
\bigl\Vert\hat{f};\mathcal{B}_{\alpha,q}\bigr\Vert=\sup_{X\times\left[1;\infty\right)}\frac{\bigl|\hat{f}\bigr|}{\mu_{\alpha,q}}\,,
\]
is finite, where the weight is given by
\[
\mu_{\alpha,q}(k,y)=\begin{cases}
{\displaystyle \frac{1}{y^{q}}\frac{1}{1+\left(\left|k\right|y\right)^{\alpha}}}\,, & q\geq0\,,\\
{\displaystyle \frac{1}{y^{q}}\frac{1}{1+\left(\left|k\right|y\right)^{\alpha}}\left(1+\frac{1}{\left(\left|k\right|y\right)^{-q}}\right)}\,, & q<0\,.
\end{cases}
\]
For $\alpha\geq0$ and $q\geq0$, the space for the velocity field
$\mathcal{U}_{\alpha,q}$ is Banach space of functions in $\mathcal{B}_{\alpha+1,q-1}$
such that the following norm is finite,\index{Space!U@$\mathcal{U}_{\alpha,q}$}
\[
\bigl\Vert\hat{f};\mathcal{U}_{\alpha,q}\bigr\Vert=\sum_{i=0}^{\lfloor q\rfloor}\sum_{j=0}^{\lfloor\alpha\rfloor}\bigl\Vert\partial_{y}^{j}\partial_{k}^{i}\hat{f};\mathcal{B}_{\alpha+1-j,q-1-i+j}\bigr\Vert\,.
\]
For $\alpha\geq1$ and $q\geq1$, the spaces for the pressure $\mathcal{P}_{\alpha,q}$
and for the nonlinearity $\mathcal{R}_{\alpha,q}$ are the Banach
spaces of functions in $\mathcal{B}_{\alpha+1,q-1}$ such that the
respective norms are finite,\index{Space!P@$\mathcal{P}_{\alpha,q}$}\index{Space!R@$\mathcal{R}_{\alpha,q}$}
\begin{align*}
\bigl\Vert\hat{f};\mathcal{P}_{\alpha,q}\bigr\Vert & =\sum_{i=0}^{\lfloor q\rfloor-1}\sum_{j=0}^{\lfloor\alpha\rfloor}\bigl\Vert\partial_{y}^{j}\partial_{k}^{i}\hat{f};\mathcal{B}_{\alpha+1-j,q-1-i+j}\bigr\Vert\,,\\
\bigl\Vert\hat{f};\mathcal{R}_{\alpha,q}\bigr\Vert & =\sum_{i=0}^{\lfloor q\rfloor-1}\sum_{j=0}^{\lfloor\alpha\rfloor-1}\bigl\Vert\partial_{y}^{j}\partial_{k}^{i}\hat{f};\mathcal{B}_{\alpha+1-j,q-1-i+j}\bigr\Vert\,.
\end{align*}

\end{defn}
\begin{rem}
The parameter $\alpha$ captures the behavior of functions at infinity
as a function of $\left|k\right|y$, which is reminiscent of the scaling
properties in $x/y$ of the Jeffery-Hamel solution. By taking the
inverse Fourier transform, the parameter $\alpha$ corresponds to
the regularity in $x$ in direct space. The index $q$ determines
the decay in $y$ at infinity. As we will see below, functions on
the boundary which are in $\mathcal{A}_{\alpha,q}$ are in the space
$\mathcal{B}_{\alpha,q}$, when evolved in time by $\e^{-\left|k\right|y}$.
The spaces $\mathcal{U}_{\alpha,q}$, $\mathcal{P}_{\alpha,q}$ and
$\mathcal{R}_{\alpha,q}$ include derivatives with respect to $k$
in order to catch the behavior near $k=0$ and derivatives with respect
to $y$ for the regularity in the $y$-direction.
\end{rem}

\begin{rem}
For $\alpha^{\prime}\geq\alpha$ and $q^{\prime}\geq q$ we have the
inclusion $\mathcal{X}_{\alpha^{\prime},q^{\prime}}\subset\mathcal{X}_{\alpha,q}$
for $\mathcal{X}=\mathcal{A}$, $\mathcal{B}$, $\mathcal{T}$, $\mathcal{W}$,
$\mathcal{U}$, $\mathcal{P}$, and $\mathcal{R}$, which will be
routinely used without mention.
\end{rem}

\begin{rem}
\label{rem:on-boundary}Since the completion defining $\mathcal{B}_{\alpha,q}$
is defined by starting from smooth functions on the closed set $\mathbb{R}\times\left[1,\infty\right)$,
the restriction of a function $\hat{f}\in\mathcal{B}_{\alpha,q}$
to the boundary $y=1$, is a function in $\mathcal{A}_{\alpha,q}$.
In the same way, the restriction of $\hat{f}\in\mathcal{R}_{\alpha,q}$
is in $\mathcal{T}_{\alpha,q}$.
\end{rem}
These spaces lead to the following regularity in direct space:
\begin{lem}
\label{lem:spaces-regularity}For $\alpha>1$ and $q\ge0$,
if $\hat{f}\in\mathcal{B}_{\alpha,q}$, we have
\begin{align*}
f & \in C(\Omega)\,, & y^{1+q}f\in & L^{\infty}(\Omega)\,, & y^{q-\varepsilon}f & \in L^{2}(\Omega)\,,
\end{align*}
for all $\varepsilon>0$. For $\alpha\notin\mathbb{N}$ and $q\ge0$,
if $\hat{f}\in\mathcal{U}_{\alpha,q}$ or $\hat{f}\in\mathcal{P}_{\alpha,q}$
we have
\begin{align*}
f & \in C^{\lfloor\alpha\rfloor}(\Omega)\,, & y^{q+i+j}\partial_{x}^{i}\partial_{y}^{j}f & \in L^{\infty}(\Omega)\,, & y^{q-1+i+j-\varepsilon}f & \in L^{2}(\Omega)\,,
\end{align*}
for $i+j\leq\lfloor\alpha\rfloor$, and $\varepsilon>0$.\end{lem}
\begin{proof}
We consider $\hat{f}\in\mathcal{B}_{\alpha,q}$. At fixed $y$, $\hat{f}(\bcdot,y)\in L^{1}(\mathbb{R})$,
so $f$ is continuous in $x$. The continuity in $y$ follows from
the fact that $\hat{f}(k,\cdot)\in C([1,\infty))$, so $f\in C(\Omega)$.
Since
\[
\big|\mathcal{F}[\hat{\mu}_{\alpha,q}]\big|\leq\frac{1}{y^{q+1}}\int_{\mathbb{R}}\frac{1}{1+z^{\alpha}}\rd z\leq\frac{\alpha}{\alpha-1}\frac{1}{y^{q+1}}\,,
\]
then $y^{q+1}f\in L^{\infty}$. Finally, by Parseval identity
\[
\int_{\mathbb{R}}\left|f(x,y)\right|^{2}\,\rd x=\int_{\mathbb{R}}\big|\hat{f}(k,y)\big|^{2}\,\rd k\leq\frac{\bigl\Vert\hat{f};\mathcal{B}_{\alpha,q}\bigr\Vert}{y^{2q+1}}\int_{\mathbb{R}}\frac{1}{1+z^{2\alpha}}\rd z\lesssim\frac{1}{y^{2q+1}}\,,
\]
so $y^{q-\varepsilon}f\in L^{2}(\Omega)$ for all $\varepsilon>0$.

Finally, we consider $\hat{f}\in\mathcal{U}_{\alpha,q}$ or $\hat{f}\in\mathcal{P}_{\alpha,q}$.
Since $\partial_{x}^{i}\partial_{y}^{j}f=\mathcal{F}\bigl[\left(\i k\right)^{i}\partial_{y}^{j}\hat{f}\bigr]$
and $\partial_{y}^{j}\hat{f}\in\mathcal{B}_{\alpha-j,q+j}$, we have
$\left|k\right|^{i}\partial_{y}^{j}\hat{f}\in\mathcal{B}_{\alpha-j-i,q+j+i}$,
so by applying the previous result, we obtain the claimed properties.
\end{proof}

\section{Stokes system\label{sec:stokes}}

In this section we consider the following inhomogeneous Stokes system,
\begin{align}
\Delta\bu-\bnabla p & =\bnabla\bcdot\mathbf{Q}\,, & \bnabla\bcdot\bu & =0\,, & \left.\bu\right|_{\partial\Omega} & =\bu^{*}\,,\label{eq:Stokes-u}
\end{align}
where $\mathbf{Q}$ is a given symmetric tensor. For simplicity we
define
\begin{align*}
R & =Q_{12}=Q_{21}\,, & S & =\frac{1}{2}\left(Q_{11}-Q_{22}\right)\,.
\end{align*}
The aim is to determine the compatibility conditions on the boundary
data $\bu^{*}$ and on the inhomogeneous term $\bR=\left(R,S\right)$,
such that \eqref{Stokes-u} admits an $(\alpha,q)$-solution:
\begin{defn}[$(\alpha,q)$-solutions for Stokes]
A pair $\left(\hat{\bu},\hat{p}\right)\in\mathcal{U}_{\alpha,q}\times\mathcal{P}_{\alpha-1,q+1}$
is called an $(\alpha,q)$-solution of the Stokes equation, if it
satisfies the Fourier transform (with respect to $x$) of the Stokes
equation \eqref{Stokes-u}.\end{defn}
\begin{lem}[$(\alpha,q)$-solutions are classical solutions]
For $\alpha>2$ and $q>0$, if $\left(\hat{\bu},\hat{p}\right)$
is an $(\alpha,q)$-solution, then its inverse Fourier transform $\left(\bu,p\right)$
has the regularity $\left(\bu,p\right)\in C^{2}(\Omega)\times C^{1}(\Omega)$
and satisfies the Stokes equation \eqref{Stokes-u} in the
classical sense.\end{lem}
\begin{proof}
In view of \lemref{spaces-regularity}, we obtain that $\left(\bu,p\right)\in C^{2}(\Omega)\times C^{1}(\Omega)$,
and since $\left(\hat{\bu},\hat{p}\right)$ satisfies the Fourier
transform of \eqref{Stokes-u}, we obtain that $\left(\bu,p\right)$
is a solution of \eqref{Stokes-u} in the classical sense.\end{proof}
\begin{thm}[existence of $(\alpha,q)$-solutions for Stokes]
\label{thm:Stokes}\index{Stokes equations!in the half-plane}\index{Asymptotic behavior!in the half-plane!Stokes solutions}We
have:
\begin{enumerate}
\item For all $\alpha>2$ and $q\geq1$, if $\hat{\mathbf{Q}}=\bzero$ and
$\hat{\bu}^{*}\in\mathcal{W}_{\alpha,q}$, there exists an $(\alpha,q)$-solution.
\item For all $\alpha>2$ and $q>1$ with $q\notin\mathbb{N}$, if $\hat{\mathbf{Q}}\in\mathcal{R}_{\alpha,q+1}$
and $\hat{\bu}^{*}\in\mathcal{T}_{\alpha,q}$ there exists an $(\alpha,q)$-solution
provided the compatibility condition $\hat{\bu}_{r}\in\mathcal{W}_{\alpha,q}$
holds, which is explicitly written in the following proposition.
\end{enumerate}
\end{thm}
\begin{prop}[compatibility conditions]
\label{prop:compatibility}For $q\in\left(1,2\right)$,
there is one compatibility condition
\begin{align*}
\hat{u}^{*}(0)+\int_{1}^{\infty}\hat{R}(0,y)\,\rd y & =0\,, & \hat{v}^{*}(0) & =0\,,
\end{align*}
and for $p\in\left(2,3\right)$, we have the additional conditions,
\begin{align*}
\partial_{k}\hat{u}^{*}(0)+\int_{1}^{\infty}\partial_{k}\hat{R}(0,y)\,\rd y-2\i\int_{1}^{\infty}\left(y-1\right)\hat{S}(0,y)\,\rd y & =0\,,\\
\partial_{k}\hat{v}^{*}(0)+\int_{1}^{\infty}\partial_{k}\hat{R}(0,y)\,\rd y+\i\int_{1}^{\infty}\left(y-1\right)\hat{R}(0,y)\,\rd y & =0\,.
\end{align*}

\end{prop}
The rest of this section is devoted to the proof of the existence
of $(\alpha,q)$-solution for Stokes system.
\begin{defn}
\label{def:operators}We define the following operators:
\begin{align*}
\left(T_{<}w\right)(k,y) & =\frac{1}{2}\int_{1}^{y}\e^{-\left|k\right|\left(y-z\right)}\left(1-\chi\left(\left|k\right|\left(z-1\right)\right)\right)w(k,z)\,\rd z\,,\\
\left(T_{>}^{\pm}w\right)(k,y) & =\frac{1}{2}\int_{y}^{\infty}\left(\e^{-\left|k\right|\left(z-y\right)}\pm\chi\left(\left|k\right|\left(z-1\right)\right)\e^{\left|k\right|\left(z-y\right)}\right)w(k,z)\,\rd z\,,\\
\left(U_{r}w\right)(k,y) & =\left(y-1\right)^{r}\e^{-\left|k\right|\left(y-1\right)}w(k)\,,
\end{align*}
where $\chi\in C_{0}^{\infty}(\mathbb{R})$ is a smooth cut-off function
such that
\begin{align}
\chi\left(\left[0,1\right]\right) & =\left\{ 1\right\} \,, & \chi\left(\left[2,\infty\right)\right) & =\left\{ 0\right\} \,,\label{eq:cutoff}
\end{align}
and their combinations:
\begin{align*}
T^{+} & =T_{<}-T_{>}^{+}\,, & T^{-} & =\sigma T_{<}+\sigma T_{>}^{-}\,,\\
B^{+} & =\left.T^{+}\right|_{y=1}\,, & B^{-} & =\left.T^{-}\right|_{y=1}\,.
\end{align*}
\end{defn}
\begin{prop}
\label{prop:Green}Formally, the Fourier transform of the
Stokes system is given by
\begin{align}
\hat{\bu} & =\mathbf{N}\hat{\bR}+\mathbf{B}\hat{\bu}_{r}\,,\qquad\hat{\bu}_{r}=\hat{\bu}^{*}-\left.\mathbf{N}\right|_{y=1}\hat{\bR}\,,\label{eq:Green-u}\\
\hat{p} & =-\i kT^{+}\hat{R}+\i kT^{-}\hat{S}-\hat{Q}_{12}-U_{0}\left[2\i k\left(\hat{u}_{r}+\sigma\hat{v}_{r}\right)\right]\,,\label{eq:Green-p}
\end{align}
where
\begin{align*}
\mathbf{N} & =\begin{pmatrix}T^{+}-T^{-}\i k\left(z-y\right) & -T^{-}-T^{+}\i k\left(z-y\right)\\
T^{+}\i k\left(z-y\right) & -T^{-}\i k\left(z-y\right)
\end{pmatrix}\,, & \mathbf{B} & =\begin{pmatrix}U_{0}-\left|k\right|U_{1} & -\i kU_{1}\\
-\i kU_{1} & U_{0}+\left|k\right|U_{1}
\end{pmatrix}\,.
\end{align*}
\end{prop}
\begin{proof}
The vorticity is
\begin{equation}
\omega=\partial_{x}v-\partial_{y}u\,,\label{eq:w-def}
\end{equation}
and the Stokes equation \eqref{Stokes-u} implies the vorticity
equation 
\begin{equation}
\Delta\omega=\left(\partial_{x}^{2}-\partial_{y}^{2}\right)R-2\partial_{x}\partial_{y}S\,.\label{eq:NS-w}
\end{equation}
By defining $\gamma=\omega+R$, the divergence-free condition, \eqref{w-def},
and \eqref{NS-w} can be rewritten as a first order differential
system in $y$,
\begin{align*}
\partial_{y}u & =\partial_{x}v-\gamma+R\,, & \partial_{y}\gamma & =\partial_{x}\eta-2\partial_{x}S\\
\partial_{y}v & =-\partial_{x}u\,, & \partial_{y}\eta & =-\partial_{x}\gamma+2\partial_{x}R\,.
\end{align*}
By taking formally the Fourier transform in the variable $x$, the
divergence-free condition, \eqref{w-def} and \eqref{NS-w}
can be rewritten as a dynamical system where $y$ plays the role of
time,
\begin{align*}
\partial_{y}\hat{\br} & =\mathbf{L}\hat{\br}+\hat{\bq}\,, & \hat{\bu}(k,1) & =\hat{\bu}_{b}\,,
\end{align*}
where
\begin{align*}
\hat{\br} & =\begin{pmatrix}\hat{u}\\
\hat{v}\\
\hat{\gamma}\\
\hat{\eta}
\end{pmatrix}\,, & \mathbf{L} & =\begin{pmatrix}0 & \i k & -1 & 0\\
-\i k & 0 & 0 & 0\\
0 & 0 & 0 & \i k\\
0 & 0 & -\i k & 0
\end{pmatrix}\,, & \hat{\bq} & =\begin{pmatrix}\hat{R}\\
0\\
-2\i k\hat{S}\\
2\i k\hat{R}
\end{pmatrix}\,.
\end{align*}
The eigenvalues of $L$ are given by $\pm\left|k\right|$, and since
we are interested in solutions with zero velocity at infinity, we
have to distinguish between stable and unstable modes, so the solution
is given by
\begin{align*}
\hat{\br}(k,y) & =\int_{1}^{y}\mathbf{P}\e^{\mathbf{L}\left(y-z\right)}\hat{\bq}(k,z)\,\rd z-\int_{y}^{\infty}\left(\mathbf{1}-\mathbf{P}\right)\e^{\mathbf{L}\left(y-z\right)}\hat{\bq}(k,z)\,\rd z+\e^{\mathbf{L}y}\hat{\br}_{s}(k)\\
\hat{\br}_{s}(k) & =\int_{1}^{\infty}\left(\mathbf{1}-\mathbf{P}\right)\e^{\mathbf{L}\left(1-z\right)}\hat{\bq}(k,z)\,\rd z+\hat{\br}_{b}(k)\,,
\end{align*}
where $\mathbf{P}$ is the projection onto stable modes and $\hat{\br}^{*}$
is such that the boundary condition in \eqref{Stokes-u}
is satisfied, 
\begin{align*}
\mathbf{P} & =\frac{1}{2\left|k\right|}\begin{pmatrix}\left|k\right| & -\i k & 1/2 & 0\\
\i k & \left|k\right| & 0 & -1/2\\
0 & 0 & \left|k\right| & -\i k\\
0 & 0 & \i k & \left|k\right|
\end{pmatrix}\,, & \hat{\br}_{b} & =\begin{pmatrix}\hat{u}_{b}\\
\hat{v}_{b}\\
2\left|k\right|\left(\hat{u}_{b}+\sigma\hat{v}_{b}\right)\\
2\i k\left(\hat{u}_{b}+\sigma\hat{v}_{b}\right)
\end{pmatrix}\,,
\end{align*}
where
\[
\sigma=\i\cdot\sign(k)\,.
\]
By using the Jordan decomposition for $\mathbf{L}$ we can explicitly
calculate the exponential and we find that
\begin{align*}
\hat{u} & =\frac{1}{2}\int_{1}^{y}\e^{-\left|k\right|\left(y-z\right)}\left(1-\left|k\right|\left(y-z\right)\right)\hat{R}_{-}(k,z)\,\rd z-\frac{1}{2}\int_{y}^{\infty}\e^{-\left|k\right|\left(z-y\right)}\left(1+\left|k\right|\left(y-z\right)\right)\hat{R}_{+}(k,z)\,\rd z\\
 & +\e^{-\left|k\right|\left(y-1\right)}\left[\left(1-\left|k\right|\left(y-1\right)\right)\hat{u}_{s}-\i k\left(y-1\right)\hat{v}_{s}\right]\,.\\
\hat{v} & =-\frac{1}{2}\int_{1}^{y}\e^{-\left|k\right|\left(y-z\right)}\i k\left(y-z\right)\hat{R}_{-}(k,z)\,\rd z+\frac{1}{2}\int_{y}^{\infty}\e^{-\left|k\right|\left(z-y\right)}\i k\left(y-z\right)\hat{R}_{+}(k,z)\,\rd z\\
 & +\e^{-\left|k\right|\left(y-1\right)}\left[\left(1+\left|k\right|\left(y-1\right)\right)\hat{v}_{s}-\i k\left(y-1\right)\hat{u}_{s}\right]\,,
\end{align*}
where $\hat{R}_{\pm}=\hat{R}\pm\sigma\hat{S}$, and
\begin{align*}
\hat{u}_{s} & =\hat{u}^{*}+\frac{1}{2}\int_{1}^{\infty}\e^{-\left|k\right|\left(z-1\right)}\left(1-\left|k\right|\left(z-1\right)\right)\hat{R}_{+}(k,z)\,\rd z\,,\\
\hat{v}_{s} & =\hat{v}^{*}+\frac{1}{2}\int_{1}^{\infty}\e^{-\left|k\right|\left(z-1\right)}\i k\left(z-1\right)\hat{R}_{+}(k,z)\,\rd z\,.
\end{align*}
By using the operators defined in \defref{operators}, we
can rewrite the integral equations as
\begin{align*}
\hat{u} & =T^{+}\left[\hat{R}-\i k\left(z-y\right)\hat{S}\right]-T^{-}\left[\hat{S}+\i k\left(z-y\right)\hat{R}\right]+U\left[\left(1-\left|k\right|\left(y-1\right)\right)\hat{u}_{r}-\i k\left(y-1\right)\hat{v}_{r}\right]\,,\\
\hat{v} & =T^{+}\left[\i k\left(z-y\right)\hat{R}\right]-T^{-}\left[\i k\left(z-y\right)\hat{S}\right]+U\left[\left(1+\left|k\right|\left(y-1\right)\right)\hat{v}_{r}-\i k\left(y-1\right)\hat{u}_{r}\right]\,,
\end{align*}
where
\begin{align*}
\hat{u}_{r} & =\hat{u}^{*}-B^{+}\left[\hat{R}-\i k\left(z-1\right)\hat{S}\right]+B^{-}\left[\hat{S}+\i k\left(z-1\right)\hat{R}\right]\,,\\
\hat{v}_{r} & =\hat{v}^{*}-B^{+}\left[\i k\left(z-1\right)\hat{R}\right]+B^{-}\left[\i k\left(z-1\right)\hat{S}\right]\,,
\end{align*}
which shows \eqref{Green-u}. Finally, from the Fourier
transform of the Stokes equation \eqref{Stokes-u}, we can
check that the pressure is effectively given by \eqref{Green-p}.
\end{proof}
In order to prove the existence of an $(\alpha,q)$-solution, we have
to estimate the operators used in \propref{Green}:
\begin{lem}
\label{lem:bound-U}For $\alpha>1$, $r\in\mathbb{N}$,
and $p\in\mathbb{R}$, the operator $U_{r}\colon\mathcal{A}_{\alpha,q}\to\mathcal{B}_{\alpha+r,q-r}$
is well-defined and continuous.\end{lem}
\begin{proof}
It suffices to prove that
\[
\frac{1}{1+\left|k\right|^{\alpha}}\left(\frac{y-1}{y}\right)^{r}\e^{-\left|k\right|\left(y-1\right)}\lesssim\frac{1}{1+\left(\left|k\right|y\right)^{\alpha+r}}\,.
\]
For $\left|k\right|y\leq1$ the result is trivial and for $\left|k\right|y>1$,
we distinguish two cases: for $y>2$, we have
\[
\frac{1}{1+\left|k\right|^{\alpha}}\left(\frac{y-1}{y}\right)^{r}\e^{-\left|k\right|\left(y-1\right)}\leq\e^{-\frac{1}{2}\left|k\right|y}\lesssim\frac{1}{1+\left(\left|k\right|y\right)^{\alpha+r}}\,,
\]
and for $1\leq y\leq2$, 
\[
\frac{1}{1+\left|k\right|^{\alpha}}\left(\frac{y-1}{y}\right)^{r}\e^{-\left|k\right|\left(y-1\right)}\lesssim\frac{1}{1+\left(\left|k\right|y\right)^{\alpha}}\frac{1}{\left(\left|k\right|y\right)^{r}}\left(\left|k\right|\left(y-1\right)\right)^{r}\e^{-\left|k\right|\left(y-1\right)}\lesssim\frac{1}{1+\left(\left|k\right|y\right)^{\alpha+r}}\,.
\]
\end{proof}
\begin{lem}
\label{lem:bound2-U}For $\alpha>1$, $r\in\mathbb{N}$
and $q\geq0$, the operator $U_{r}\colon\mathcal{W}_{\alpha,q}\to\mathcal{U}_{\alpha+r,q-r}$
is well-defined and continuous.\end{lem}
\begin{proof}
First of all, we have
\[
\partial_{y}U_{r}=U_{r-1}+\left|k\right|U_{r}\,,
\]
so $\partial_{y}U_{r}:\mathcal{A}_{\alpha+1,q-1}\to\mathcal{B}_{\alpha+r,q-r}$.
For $q>0$, we have 
\[
\partial_{k}\left(U_{r}w\right)=U_{r}\partial_{k}w+\i\sigma U_{r+1}w\,,
\]
so by \lemref{bound-U} we obtain that $\partial_{k}\left(U_{r}w\right)\in\mathcal{B}_{\alpha+r+1,q-r-2}$.
The result now follows by a recursion on the number of derivatives.\end{proof}
\begin{lem}
For $\alpha\geq0$ and $q\geq0$, the operator $T_{<}\colon\mathcal{B}_{\alpha,q}\to\mathcal{B}_{\alpha+1,q-1}$
is well-defined and continuous.\end{lem}
\begin{proof}
Due to the cut-off function, the integral vanishes for $\left|k\right|\left(y-1\right)\leq1$,
and for $\left|k\right|y\geq1+\left|k\right|$, we split the integral:
\begin{align*}
\left(T_{<}\mu_{\alpha,q}\right)(k,y) & \leq\int_{1}^{\frac{y+1}{2}}\e^{-\left|k\right|\left(y-z\right)}\left(1-\chi\left(\left|k\right|\left(z-1\right)\right)\right)\mu_{\alpha,q}(k,z)\,\rd z+\int_{\frac{y+1}{2}}^{y}\e^{-\left|k\right|\left(y-z\right)}\mu_{\alpha,q}(k,z)\,\rd z\\
 & \lesssim\e^{-\left|k\right|\left(y-1\right)/2}\int_{1}^{\frac{y+1}{2}}\eta_{\alpha,q}(k)\,\rd z+\mu_{\alpha,q}(k,y)\int_{\frac{y+1}{2}}^{y}\e^{-\left|k\right|\left(y-z\right)}\,\rd z\\
 & \lesssim\e^{-\left|k\right|\left(y-1\right)/2}\left(y-1\right)\eta_{\alpha,q}(k)+\frac{1}{\left|k\right|y}\mu_{\alpha,q-1}\left(k,y\right)\lesssim\mu_{\alpha+1,q-1}(k,y)\,,
\end{align*}
where for the last step we apply \lemref{bound-U}.\end{proof}
\begin{lem}
\label{lem:bound-T}For all $\alpha>1$ and $q\geq0$ with
$q\neq1$, the operators $T_{>}^{\pm}\colon\mathcal{B}_{\alpha,q}\to\mathcal{B}_{\alpha+1,q-1}$
are well-defined and continuous.\end{lem}
\begin{proof}
For $\left|k\right|y>1$, we have
\begin{align*}
\left(T_{>}^{\pm}\mu_{\alpha,q}\right)(k,y) & \lesssim\int_{y}^{\infty}\e^{-\left|k\right|\left(z-y\right)}\mu_{\alpha,q}(k,z)\,\rd z\leq\mu_{\alpha,q}\left(k,y\right)\int_{y}^{\infty}\e^{\left|k\right|\left(y-z\right)}\,\rd z\\
 & \lesssim\frac{1}{\left|k\right|y}\mu_{\alpha,q-1}(k,y)\lesssim\mu_{\alpha+1,q-1}(k,y)\,,
\end{align*}
and for $\left|k\right|y<1$, since $p\neq q$, we have
\[
\left(T_{>}^{\pm}\mu_{\alpha,q}\right)(k,y)\lesssim\int_{y}^{\infty}\mu_{\alpha,q}(k,z)\,\rd z\leq\left|k\right|^{q-1}\int_{\left|k\right|y}^{\infty}\frac{1}{u^{q}}\frac{1}{1+u^{\alpha}}\rd u\lesssim\mu_{\alpha+1,q-1}(k,y)\,.
\]
\end{proof}
\begin{lem}
\label{lem:bound2-T}For all $\alpha>0$ and $q\geq0$ with
$q\notin\mathbb{N}$, the operators $T_{\pm}\colon\mathcal{R}_{\alpha,q+1}\to\mathcal{U}_{\alpha+1,q}$
are well-defined and continuous.\end{lem}
\begin{proof}
First, by using \lemref{bound-T}, we have $T_{\pm}\colon\mathcal{B}_{\alpha+1,q}\to\mathcal{B}_{\alpha+2,q-1}$.
By take the derivative with respect to $y$, we get
\begin{align*}
\partial_{y}\left(T_{<}q\right)(k,y) & =\frac{1}{2}\left(1-\chi\left(\left|k\right|\left(y-1\right)\right)\right)q(k,y)-\left|k\right|\left(T_{<}q\right)(k,y)\,,\\
\partial_{y}\left(T_{>}^{\pm}q\right)(k,y) & =\frac{-1}{2}\left(1\pm\chi\left(\left|k\right|\left(y-1\right)\right)\right)q(k,y)+\left|k\right|\left(T_{>}^{\mp}q\right)(k,y)\,,
\end{align*}
so the time-derivative of the operators are
\begin{align*}
\partial_{y}T^{+} & =\frac{1}{2}+\i kT^{-}\,, & \partial_{y}T^{-} & =-\i kT^{+}\,,
\end{align*}
so $\partial_{y}T_{\pm}\colon\mathcal{B}_{\alpha+1,q}\to\mathcal{B}_{\alpha+1,q}$.
Since the integrand of $T_{<}^{-}$ vanishes at $k=0$, we have
\begin{align*}
\partial_{k}\left(T_{<}w\right) & =T_{<}\partial_{k}w+\i\left(y-1\right)\sigma T_{<}w-\i\sigma\tilde{T}_{<}\left(z-1\right)w\,,\\
\partial_{k}\left(T_{>}^{+}w\right) & =T_{>}^{+}\partial_{k}w-\i\left(y-1\right)\sigma T_{>}^{-}w+\i\sigma\tilde{T}_{>}^{-}\left(z-1\right)w\,,\\
\partial_{k}\left(\sigma T_{>}^{-}w\right) & =\sigma T_{>}^{-}\partial_{k}w+\i\left(y-1\right)T_{>}^{+}w-\i\tilde{T}_{>}^{+}\left(z-1\right)w\,,
\end{align*}
where a tilde over an operator denotes the same operator where $\chi$
is replaced by $\chi+\chi^{\prime}$ which is also a cut-off function
satisfying \eqref{cutoff}. Therefore,
\begin{align*}
\partial_{k}\left(T_{+}w\right) & =T_{+}\partial_{k}w+\i\left(y-1\right)T_{-}w-\i\sigma\tilde{T}_{-}\left(z-1\right)w\,,\\
\partial_{k}\left(T_{-}w\right) & =T_{-}\partial_{k}w-\i\left(y-1\right)T_{+}w+\i\tilde{T}_{+}\left(z-1\right)w\,,
\end{align*}
and by using the previously shown properties on the operators $T_{\pm}$,
we obtain that $\partial_{k}\left(T_{\pm}w\right)\in\mathcal{B}_{\alpha+2,q-2}$.
By recursion on the number of derivatives we obtain $T_{\pm}w\in\mathcal{U}_{\alpha+1,q}$.
\end{proof}
We can now apply these lemmas to prove the existence of $(\alpha,q)$-solutions:
\begin{proof}[Proof of \thmref{Stokes}]
By applying \lemref{bound2-U}, we have $\mathbf{B}\colon\mathcal{W}_{\alpha,p}\to\mathcal{U}_{\alpha,p}$,
and therefore, if $\hat{\mathbf{Q}}=\bzero$ and $\hat{\bu}^{*}\in\mathcal{W}_{\alpha,q}$,
we obtain that $\left(\hat{\bu},\hat{p}\right)\in\mathcal{U}_{\alpha,q}\times\mathcal{P}_{\alpha-1,q+1}$.
By applying \lemref{bound2-T}, noting that
\[
z-y=\left(z-1\right)-\left(y-1\right)\,,
\]
and bounding each resulting term separately, we obtain that $\mathbf{N}\colon\mathcal{R}_{\alpha,q+1}\to\mathcal{U}_{\alpha,q}$,
for $q>1$ with $q\notin\mathbb{N}$. In view of \remref{on-boundary},
we have $\left.\mathbf{N}\right|_{y=1}\colon\mathcal{R}_{\alpha,q+1}\to\mathcal{T}_{\alpha,q}$.
In case $\hat{\bu}_{r}\in\mathcal{W}_{\alpha,q}$, since $\mathbf{B}\colon\mathcal{W}_{\alpha,p}\to\mathcal{U}_{\alpha,p}$,
we have $\left(\hat{\bu},\hat{p}\right)\in\mathcal{U}_{\alpha,q}\times\mathcal{P}_{\alpha-1,q+1}$.
\end{proof}
The deduction of the compatibility conditions is now straightforward:
\begin{proof}[Proof of \propref{compatibility}]
Since $\hat{\bu}_{r}\in\mathcal{T}_{\alpha,p}$, we use the characterization
of $\mathcal{W}_{\alpha,p}$ in terms of elements of $\mathcal{T}_{\alpha,p}$
provided in \lemref{spaces-TtoU}. The first compatibility
condition is $\hat{\bu}_{r}(0)=\bzero$, and the second $\partial_{k}\hat{\bu}_{r}(0)=\bzero$.
By explicit calculations, we obtain the claimed conditions.
\end{proof}

\section{Strong solutions to the Navier-Stokes equation\label{sec:ns}}

The Navier-Stokes equation in the half-space can be written as the
Stokes system \eqref{Stokes-u} with $\mathbf{Q}=\bu\otimes\bu$,
and we are going to look for solutions of the form $\hat{\bu}=\hat{\bu}_{\Phi}^{\sigma}+\hat{\bu}_{1}$
and perform a fixed point argument on $\hat{\bu}_{1}\in\mathcal{U}_{\alpha,q}$.
First of all, the Jeffery-Hamel solution \eqref{JH-xy}
at fixed values of $y$ and large values of $\pm s$, where $s=x/y$,
is
\begin{align*}
u_{\Phi}^{\sigma}(x,y) & \approx\frac{-1}{y\, s^{2}}f^{\prime}\!\left(\pm\tfrac{\pi}{2}\right)\,, & v_{\Phi}^{\sigma}(x,y) & \approx\frac{-1}{y\, s^{3}}f^{\prime}\!\left(\pm\tfrac{\pi}{2}\right)\,,
\end{align*}
so that its Fourier transforms satisfies $\partial_{y}^{i}\hat{\bu}_{\Phi}^{\sigma}\in\mathcal{B}_{\alpha,i}$,
so
\begin{equation}
\hat{\bu}_{\Phi}^{\sigma}\in\mathcal{U}_{\alpha,0}\,,\label{eq:space-JH}
\end{equation}
for arbitrary $\alpha>1$. In order to treat the nonlinearity $\bu_{\Phi}^{\sigma}\otimes\bu_{1}$,
we need the following proposition concerning the convolution:
\begin{prop}
\label{prop:convolution}For $\alpha>1$ and $q\geq1$ the
convolution $*\colon\mathcal{U}_{\alpha,1}\times\mathcal{U}_{\alpha,q}\to\mathcal{R}_{\alpha,q+1}$
is a continuous bilinear map.\index{Convolution}\end{prop}
\begin{proof}
First, we show that the map $*\colon\mathcal{B}_{\alpha,p}\times\mathcal{B}_{\alpha,q-1}\to\mathcal{B}_{\alpha,p+q}$
is a continuous bilinear map, for $p,q\geq0$. If $\hat{f}\in\mathcal{B}_{\alpha,p}$
and $\hat{g}\in\mathcal{B}_{\alpha,q-1}$, $\hat{f}$ is in $L^{\infty}(\mathbb{R})$
and $\hat{g}$ is in $L^{1}(\mathbb{R})$ for fixed $y\in\left[1;\infty\right)$,
so \citep[see for example][Proposition 8.8]{Folland1999} $\hat{f}*\hat{g}\in C(\mathbb{R})$.
The dependence of the convolution on the power of $y$ is trivial
and it therefore suffices to prove that
\[
\mu_{\alpha,0}*\mu_{\alpha,0}^{\nu}\lesssim\mu_{\alpha,1}\,,
\]
for $\nu\in\left(0,1\right)$, where
\[
\mu_{\alpha,0}^{\nu}(k,y)=\frac{1}{\left(\left|k\right|y\right)^{\nu}}\frac{1}{1+\left(\left|k\right|y\right)^{\alpha-\nu}}\,.
\]
For $\left|k\right|y\leq1$, we have
\begin{align*}
\int_{\mathbb{R}}\mu_{\alpha,0}\left(\ell,y\right)\mu_{\alpha,0}^{\nu}\left(k-\ell,y\right)\rd\ell & \leq\int_{\mathbb{R}}\mu_{\alpha,0}^{\nu}\left(k-\ell,y\right)\rd\ell\\
 & \leq\frac{1}{y}\int_{\mathbb{R}}\mu_{\alpha,0}^{\nu}\left(\ell,1\right)\rd\ell\\
 & \lesssim\mu_{\alpha,1}\left(k,y\right)\,,
\end{align*}
and, for $\left|k\right|y>1$, we have that $\mu_{\alpha,0}^{\nu}\leq\mu_{\alpha,0}$
and therefore, by splitting the integral at $k/2$, we find that
\begin{align*}
\int_{\mathbb{R}}\mu_{\alpha,0}\left(\ell,y\right)\mu_{\alpha,0}^{\nu}\left(k-\ell,y\right)\rd\ell & \leq\mu_{\alpha,0}\left(k/2,y\right)\int_{\mathbb{R}}\mu_{\alpha,0}\left(\ell,y\right)\rd\ell\\
 & \leq\mu_{\alpha,0}\left(k/2,y\right)\frac{1}{y}\int_{\mathbb{R}}\mu_{\alpha,0}\left(\ell,1\right)\rd\ell\\
 & \lesssim\mu_{\alpha,1}\left(k,y\right)\,.
\end{align*}

Now we consider $\hat{f}\in\mathcal{U}_{\alpha,1}$, and $\hat{g}\in\mathcal{U}_{\alpha,q}$.
If $\alpha>0$, we have \citep[see for example][Exercise 8.8]{Folland1999}
$\partial_{k}\left(\hat{f}*\hat{g}\right)=\hat{f}*\partial_{k}\hat{g}$,
so by using the previous result, $\hat{f}*\partial_{k}\hat{g}\in\mathcal{B}_{\alpha+1,q-1}$.
By taking the derivative with respect to $y$, we have $\partial_{y}\left(\hat{f}*\hat{g}\right)=\partial_{y}\hat{f}*\hat{g}+\hat{f}*\partial_{y}\hat{g}\in\mathcal{B}_{\alpha,q+1}$.
Finally, by a recursion on the number of derivatives, we obtain that
$\hat{f}*\hat{g}\in\mathcal{R}_{\alpha,q+1}$.
\end{proof}
Now we can state the main theorem:
\begin{thm}[existence of $(\alpha,q)$-solutions for Navier-Stokes]
\label{thm:Navier-Stokes}\index{Navier-Stokes equations!in the half-plane!asymptotic behavior}\index{Navier-Stokes equations!in the half-plane!strong solutions}\index{Existence!strong solutions!in the half-plane}\index{Strong solutions!in the half-plane}For
$\alpha>2$ and $q\in\left(1,2\right)$, there exists $\nu>0$ such
that for any $\Phi\in\mathbb{R}$ and $\hat{\bu}_{s}\in\mathcal{T}_{\alpha,q}$
satisfying
\begin{align*}
\left|\Phi\right| & \leq\nu\,, & \left\Vert \hat{\bu}_{s};\mathcal{T}_{\alpha,0}\right\Vert  & \leq\nu\,, & \int_{\mathbb{R}}\bu_{s}(x)\,\rd x & =\bzero\,,
\end{align*}
there exists $A\in\mathbb{R}$ such that there exists $\left(\bu,p\right)\in C^{2}(\Omega)\times C^{1}(\Omega)$
satisfying \eqref{ns} with
\[
\bu^{*}(x)=\bu_{\Phi}^{\sigma}(x,1)+\frac{\left(A,0\right)}{\sqrt{2\pi}}\e^{-\frac{1}{2}x^{2}}+\bu_{s}(x)\,.
\]
Moreover, $\hat{\bu}-\hat{\bu}_{\Phi}^{\sigma}\in\mathcal{U}_{\alpha,q}$
so that
\[
\lim_{y\to\infty}y\left(\sup_{x\in\mathbb{R}}\left|\bu-\bu_{\Phi}^{\sigma}\right|\right)=0\,,
\]
and
\begin{align}
\begin{aligned}\bnabla\bu & \in L^{2}(\Omega)\,,\\
\bu/y & \in L^{2}(\Omega)\,,
\end{aligned}
 &  & \begin{aligned}y\bu & \in L^{\infty}(\Omega)\,,\\
y^{2}\bnabla\bu & \in L^{\infty}(\Omega)\,.
\end{aligned}
\label{eq:asol-bounds-ns}
\end{align}
\end{thm}
\begin{proof}
We look for solutions of the form $\hat{\bu}=\hat{\bu}_{\Phi}^{\sigma}+\hat{\bu}_{1}$
and perform a fixed point argument on $\hat{\bu}_{1}$ in the space
$\hat{\bu}_{1}\in\mathcal{U}_{\alpha,q}$. In view of the previous
section, the Navier-Stokes equation can be written as the Stokes equation
\eqref{Stokes-u} for $\bu_{1}$ where $\mathbf{Q}=\bu_{\Phi}^{\sigma}\otimes\bu_{1}+\bu_{1}\otimes\bu_{\Phi}^{\sigma}+\bu_{1}\otimes\bu_{1}$.
The boundary condition is
\[
\hat{\bu}^{*}(k)=\hat{\bu}_{\Phi}^{\sigma}(k,1)+\left(A,0\right)\e^{-\frac{1}{2}k^{2}}+\hat{\bu}_{s}(k)\,,
\]
and the compatibility conditions of  \propref{compatibility}
are given by
\begin{align*}
\hat{u}_{r}(0) & =A+\int_{1}^{\infty}\hat{R}(0,y)\,\rd y\,, & \hat{v}_{r}(0) & =0\,,
\end{align*}
since by hypothesis $\hat{\bu}_{s}(k)=\bzero$. Therefore, by defining
$A=-\int_{1}^{\infty}\hat{R}(0,y)\,\rd y$, the two compatibility
conditions are fulfilled. In what follows, $C>0$ represents a generic
constant depending on $q$, but not on $\varepsilon$. By \propref{convolution},
we have $\hat{\mathbf{Q}}\in\mathcal{R}_{\alpha,q+1}$ and
\[
\left\Vert \hat{\bR};\mathcal{R}_{\alpha,q+1}\right\Vert \leq C\left(\left\Vert \hat{\bu}_{\Phi}^{\sigma};\mathcal{U}_{\alpha,0}\right\Vert +\left\Vert \hat{\bu}_{1};\mathcal{U}_{\alpha,q}\right\Vert \right)\left\Vert \hat{\bu}_{1};\mathcal{U}_{\alpha,q}\right\Vert \,.
\]
Since
\[
\left|A\right|\leq\left|\int_{1}^{\infty}\hat{R}(0,y)\,\rd y\right|\leq\frac{1}{p}\left\Vert \hat{\bR};\mathcal{R}_{\alpha,q+1}\right\Vert \,,
\]
we have
\[
\left\Vert \hat{\bu}^{*}-\hat{\bu}_{\Phi}^{\sigma};\mathcal{T}_{\alpha,q}\right\Vert \leq\left\Vert \hat{\bu}_{s};\mathcal{T}_{\alpha,q}\right\Vert +C\left\Vert \hat{\bR};\mathcal{R}_{\alpha,q+1}\right\Vert \,.
\]
By applying \thmref{Stokes}, we obtain that 
\[
\left\Vert \hat{\bu}_{1};\mathcal{U}_{\alpha,q}\right\Vert \leq C\left\Vert \hat{\bu}_{s};\mathcal{T}_{\alpha,q}\right\Vert +C\left(\left\Vert \hat{\bu}_{\Phi}^{\sigma};\mathcal{U}_{\alpha,0}\right\Vert +\left\Vert \hat{\bu}_{1};\mathcal{U}_{\alpha,q}\right\Vert \right)\left\Vert \hat{\bu}_{1};\mathcal{U}_{\alpha,p}\right\Vert \,.
\]
Therefore, for $\varepsilon>0$ small enough, a fixed point argument
shows the existence of a solution $\left(\hat{\bu},\hat{p}\right)\in\mathcal{U}_{\alpha,q}\times\mathcal{P}_{\alpha-1,p+1}$
of the Fourier transform of the Navier-Stokes equation. In the same
way as in \thmref{Stokes}, we obtain the claimed regularity
and the asymptotic properties.
\end{proof}

\section{Existence of weak solutions\label{sec:weaksol}}

In this section we define weak solutions for our problem, and we discuss
in particular the technicalities due to the inhomogeneous boundary
conditions on an unbounded boundary. In order to show that our definition
of weak solutions is general enough, we then construct such solutions
by Leray's method. To study an inhomogeneous boundary problem, it
is standard \citep[see for example][Chapter 5.]{Ladyzhenskaya-MathematicalTheory1963}
to define weak solutions by using an extension map to write the energy
inequality.

We denote by $D_{\sigma}^{1,2}(\Omega)$ the subspace of the homogeneous
Sobolev space of order $(1,2)$ of divergence-free functions on $\Omega$,
and by $D_{0,\sigma}^{1,2}(\Omega)$ the completion with respect to
the norm of $D_{\sigma}^{1,2}(\Omega)$ of the vector space of smooth
divergence-free functions with compact support in $\Omega$. We refer
the reader to \citep[Chapter II.6.]{Galdi-IntroductiontoMathematical2011}
for the properties of these spaces. The main tool in studying the
existence and uniqueness of weak solution is the Hardy inequality:
\begin{prop}[{\citealp[\S 2.7.1]{Mazya-SobolevSpaceswith2011}}]
For all $\bu\in D_{0,\sigma}^{1,2}(\Omega)$, we have\index{Hardy inequality!in the half-plane}\index{Inequality!Hardy!in the half-plane}
\[
\left\Vert \bu/y\right\Vert _{2}\leq2\left\Vert \bnabla\bu\right\Vert _{2}\,.
\]

\end{prop}
We define an extension map as follows, whose existence is proved in
\thmref{Stokes} for $\hat{\bu}^{*}\in\mathcal{A}_{\alpha,0}$:
\begin{defn}[extension]
\label{def:extension}Given a boundary condition $\bu^{*}$,
an extension is a map $\ba\in D_{\sigma}^{1,2}(\Omega)$ such that
$\ba/y\in L^{2}(\Omega)$, $y\ba\in L^{\infty}(\Omega)$ and $y^{2}\bnabla\ba\in L^{\infty}$
and such that the trace of $\ba$ on $\partial\Omega$ is $\bu^{*}$.
\end{defn}

\begin{defn}[weak solution]
\label{def:weak-solution}\index{Weak solutions!in the half-plane!definition}A
weak solution in the domain $\Omega$ with boundary condition $\bu^{*}$
is a vector field $\bu=\ba+\bv$, where $\ba$ is an extension of
$\bu^{*}$ and $\bv\in D_{0,\sigma}^{1,2}(\Omega)$ which satisfies:
\begin{equation}
\int_{\Omega}\bnabla\bu:\bnabla\bphi+\int_{\Omega}\left(\bu\bcdot\bnabla\bu\right)\bcdot\bphi=0\,,\label{eq:weak-sol}
\end{equation}
for arbitrary smooth divergence-free vector-fields $\bphi$ with compact
support in $\Omega$.
\end{defn}
The main result of this section is the existence of weak solutions:
\begin{thm}[existence of weak solution]
\label{thm:existence-weak}\index{Navier-Stokes equations!in the half-plane!weak solutions}\index{Existence!weak solutions!in the half-plane}\index{Weak solutions!in the half-plane!existence}For
a small enough boundary condition $\bu^{*}$ (more precisely such
that $\left\Vert y\ba\right\Vert _{\infty}+\left\Vert \ba/y\right\Vert _{2}$
is small enough), there exists a weak solution $\bu$ in $\Omega$\textup{.}
\end{thm}
Before proving this theorem, we mention the fact that any weak solution
vanishes at infinity in the following sense:
\begin{prop}
If $\bu=\ba+\bv$ is a weak solution with $\bv\in D_{0,\sigma}^{1,2}(\Omega)$,
and $y\ba\in L^{\infty}(\Omega)$, then\index{Weak solutions!in the half-plane!limit of the velocity}\index{Limit of the velocity!in the half-plane}\index{Asymptotic behavior!in the half-plane!Navier-Stokes solutions!weak
solutions}
\[
\lim_{\left|\bx\right|\to\infty}\bu=\bzero\,,
\]
in the following sense
\[
\lim_{r\to\infty}\int_{-\pi/2}^{\pi/2}\left|\bu(r\sin\theta,1+r\cos\theta)\right|^{2}\rd\theta=0\,.
\]
\end{prop}
\begin{proof}
First of all, by Hardy inequality, we have $\bu/y\in L^{2}(\Omega)$.
We define the half-ball $\Omega_{n}$ and the half-shell $S_{n}$
by
\begin{align*}
\Omega_{n} & =B((0,1),n)\cap\Omega\,, & S_{n} & =\Omega_{2n}\setminus\Omega_{n}\,,
\end{align*}
with $B((0,1),n)$ the open ball of radius $n$ centered at $(0,1)$.
By using the trace theorem in $S_{1}$, there exists $C>0$ such that
\[
\left\Vert \bu;L^{2}(\partial\Omega_{1})\right\Vert ^{2}\leq\left\Vert \bu;L^{2}(\partial S_{1})\right\Vert ^{2}\leq C\left\Vert \bu;L^{2}(S_{1})\right\Vert ^{2}+C\left\Vert \bnabla\bu;L^{2}(S_{1})\right\Vert ^{2}\,.
\]
By a rescaling argument, we obtain that
\[
\frac{1}{n}\left\Vert \bu;L^{2}(\partial\Omega_{n})\right\Vert ^{2}\leq\frac{C}{n^{2}}\left\Vert \bu;L^{2}(S_{n})\right\Vert ^{2}+C\left\Vert \bnabla\bu;L^{2}(S_{n})\right\Vert ^{2}\,,
\]
and since $y\leq n$ in $\Omega_{n}$, we have
\[
\frac{1}{n}\left\Vert \bu;L^{2}(\partial\Omega_{n})\right\Vert ^{2}\leq C\left\Vert \bu/\by;L^{2}(S_{n})\right\Vert ^{2}+C\left\Vert \bnabla\bu;L^{2}(S_{n})\right\Vert ^{2}\,.
\]
In the limit $n\to\infty$, the right hand-side converges to zero,
because $\bu/y,\bnabla\bu\in L^{2}(\Omega)$ and since the integrals
over $S_{n}$ can be written as the difference of integrals over $\Omega_{2n}$
and \textbf{$\Omega_{n}$}. Finally,
\[
\int_{-\pi/2}^{\pi/2}\left|\bu(r\sin\theta,1+r\cos\theta)\right|^{2}\rd\theta=\frac{1}{2\pi n}\left\Vert \bu;L^{2}(\partial\Omega_{n})\right\Vert \,,
\]
and the result is proved.
\end{proof}
As usual, to show the existence of a weak solution in an unbounded
domain, we first prove, for arbitrary $n\in\mathbb{N}$, the existence
of a weak solution in the domains $\Omega_{n}$ defined in the previous
proof. To this end, we introduce the concept of approximate weak solution
in $\Omega_{n}$ and then apply the Leray-Schauder theorem to prove
the existence of such approximate solutions.
\begin{defn}[approximate weak solution]
\label{def:approx-weak-sol}For $n\in\mathbb{N}$, an approximate
weak solution is a vector field $\bu_{n}=\ba+\bv_{n}$ where $\bv_{n}\in D_{0,\sigma}^{1,2}(\Omega)$
with support in $\Omega_{n}$, which satisfies
\[
\int_{\Omega}\bnabla\bv_{n}:\bnabla\bphi+\int_{\Omega}\left(\bu_{n}\bcdot\bnabla\bu_{n}\right)\bcdot\varphi=0\,,
\]
for arbitrary smooth divergence-free vector-fields $\bphi$ with support
in $\Omega_{n}$.\end{defn}
\begin{lem}[existence of approximate weak solution]
\label{lem:existence-approx-weak-sol}Provided $\bu^{*}$
is small enough, there exists for all $n\in\mathbb{N}$ an approximate
weak solution $\bu_{n}=\ba+\bv_{n}$, with $\left\Vert \bnabla\bv_{n}\right\Vert \leq1$.\end{lem}
\begin{proof}
First we note that the trilinear term can be bounded as
\[
\left|\int_{\Omega}\left(\bu_{n}\bcdot\bnabla\bu_{n}\right)\bcdot\bphi\right|\leq\left\Vert \bnabla\bu_{n}\right\Vert _{2}\left\Vert \bu_{n}\right\Vert _{4}\left\Vert \bphi\right\Vert _{4}\,.
\]
Therefore the map
\begin{align*}
H_{0,\sigma}^{1}(\Omega_{n}) & \to\mathbb{R}\\
\bphi & \mapsto-\int_{\Omega}\left(\bu_{n}\bcdot\bnabla\bu_{n}\right)\bcdot\bphi\,,
\end{align*}
is a continuous linear form, and by the Riesz representation theorem,
there exists a map $F_{n}:W_{0,\sigma}^{1,2}(\Omega_{n})\to W_{0,\sigma}^{1,2}(\Omega_{n})$
such that
\[
\left(F_{n}(\bv_{n}),\bphi\right)=-\int_{\Omega}\left(\bu_{n}\bcdot\bnabla\bu_{n}\right)\bcdot\bphi\,.
\]
The map $F_{n}$ is continuous on $W_{0,\sigma}^{1,2}(\Omega_{n})$
when equipped with the $L^{4}$-norm, and, since $W_{0,\sigma}^{1,2}(\Omega_{n})$
is compactly embedded in $L^{4}(\Omega_{n})$, $F_{n}$ is completely
continuous.

The problem of finding an approximate solution is equivalent to solving
the equation
\[
\bv_{n}=F_{n}(\bv_{n})
\]
in $W_{0,\sigma}^{1,2}(\Omega_{n})$. From the Leray-Schauder fixed
point theorem \citep[see for example][Theorem 11.6.]{Gilbarg.Trudinger-EllipticPartialDifferential1988}
to prove the existence of an approximate weak solution it is sufficient
to prove that the set of all possible solutions of the equation
\begin{equation}
\bv_{n}=\lambda F_{n}(\bv_{n})\,,\label{eq:weak-lambda}
\end{equation}
is uniformly bounded in $\lambda\in\left[0,1\right]$.

To this end, we take the scalar product of \eqref{weak-lambda}
with $\bv_{n}$, and after integrations by parts, we get 
\[
\int_{\Omega}\bnabla\bv_{n}:\bnabla\bv_{n}=\lambda\int_{\Omega}\left(\bu_{n}\bcdot\bnabla\bv_{n}\right)\bcdot\ba\,.
\]
Therefore by Hölder inequality, we obtain
\[
\left\Vert \bnabla\bv_{n}\right\Vert _{2}^{2}\leq\lambda\left(\left\Vert \ba/y\right\Vert _{2}+\left\Vert \bv_{n}/y\right\Vert _{2}\right)\left\Vert \bnabla\bv_{n}\right\Vert _{2}\left\Vert y\ba\right\Vert _{\infty}\,,
\]
and therefore by using Hardy inequality,
\[
\left\Vert \bnabla\bv_{n}\right\Vert _{2}\leq\lambda\left(\left\Vert \ba/y\right\Vert _{2}+2\left\Vert \bnabla\bv_{n}\right\Vert _{2}\right)\left\Vert y\ba\right\Vert _{\infty}\,.
\]
For $\lambda\in\left[0,1\right]$, we finally obtain for $n$ big
enough, and $\ba$ small enough,
\[
\left\Vert \bnabla\bv_{n}\right\Vert _{2}\leq\frac{\left\Vert \ba/y\right\Vert _{2}\left\Vert y\ba\right\Vert _{\infty}}{1-2\left\Vert y\ba\right\Vert _{\infty}}\,,
\]
which proves that $\bnabla\bv_{n}$ is uniformly bounded.
\end{proof}
We are now able to take the limit $n\to\infty$ and prove the existence
of a weak solution in $\Omega$:
\begin{proof}[Proof of \thmref{existence-weak}]
By \lemref{existence-approx-weak-sol}, there exists for
any $n\in\mathbb{N}$ an approximate weak-solution $\bv_{n}$ and
the sequence $\left(\bv_{n}\right)_{n\in\mathbb{N}}$ is bounded in
$D_{0,\sigma}^{1,2}(\Omega)$. Therefore, we can extract a subsequence,
denoted also by $\left(\bv_{n}\right)_{n\in\mathbb{N}}$, which converges
weakly to $\bv$ in $D_{0,\sigma}^{1,2}(\Omega)$. Now let $\bphi$
be a test function with compact support in $\Omega$. Then, there
exists $m\in\mathbb{N}$ such that the support of $\bphi$ is in $\Omega_{m}$.
Therefore, we have for any $n\geq m$,
\[
\int_{\Omega}\bnabla\bv_{n}:\bnabla\bphi+\int_{\Omega}\left(\bu_{n}\bcdot\bnabla\bu_{n}\right)\bcdot\bphi=0\,.
\]
By replacing the test function by $\bv_{n}$ and after integration
by parts, we obtain
\[
\int_{\Omega}\bnabla\bv_{n}:\bnabla\bv_{n}\leq\int_{\Omega}\left(\bu_{n}\bcdot\bnabla\bv_{n}\right)\bcdot\ba\,.
\]
Therefore it remains to prove that these last two equations remain
valid in the limit $n\to\infty$. By definition of the weak convergence,
we have
\[
\lim_{n\to\infty}\int_{\Omega}\bnabla\bv_{n}:\bnabla\bphi=\int_{\Omega}\bnabla\bv:\bnabla\bphi\,,
\]
and
\[
\int_{\Omega}\bnabla\bv:\bnabla\bv\leq\liminf_{n\to\infty}\int_{\Omega}\bnabla\bv_{n}:\bnabla\bv_{n}\leq1\,.
\]
Since $\varphi$ has support in $\Omega_{m}$,
\begin{align*}
\left|\int_{\Omega}\left(\bu_{n}\bcdot\bnabla\bu_{n}-\bu\bcdot\bnabla\bu\right)\bcdot\bphi\right| & =\left|\int_{\Omega_{m}}\left(\left(\bu_{n}-\bu\right)\bcdot\bnabla\bu_{n}+\bu\bcdot\bnabla\left(\bu_{n}-\bu\right)\right)\bcdot\bphi\right|\\
 & \leq\left|\int_{\Omega_{m}}\left(\left(\bv_{n}-\bv\right)\bcdot\bnabla\bu_{n}\right)\bcdot\bphi\right|+\left|\int_{\Omega_{m}}\left(\bu\bcdot\bnabla\bphi\right)\bcdot\left(\bv_{n}-\bv\right)\right|\\
 & \leq\left(\left\Vert \bnabla\bu_{n}\right\Vert _{2}\left\Vert \bphi\right\Vert _{\infty}+2\left\Vert \bnabla\bu_{n}\right\Vert _{2}\left\Vert y\bnabla\bphi\right\Vert _{\infty}\right)\left\Vert \bv_{n}-\bv;L_{2}(\Omega_{m})\right\Vert \,,
\end{align*}
and therefore since $D_{0,\sigma}^{1,2}(\Omega_{m})$ is compactly
embedded in $L_{2}(\Omega_{m})$ this proves that
\[
\lim_{n\to\infty}\int_{\Omega}\left(\bu_{n}\bcdot\bnabla\bu_{n}\right)\bcdot\bphi=\int_{\Omega}\left(\bu\bcdot\bnabla\bu\right)\bcdot\bphi\,.
\]

\end{proof}

\section{Uniqueness\label{sec:uniqueness}}

In this section, we prove a weak-strong uniqueness theorem by exploiting
the properties of $(\alpha,q)$-solutions. Namely we prove that any
weak solution satisfying the decay properties \eqref{asol-bounds-ns}
of an $(\alpha,q)$-solution coincides with any weak-solutions for
the same boundary data. The ideas of the proof that are not specific
to the presence of an extension can be found in \citet{Hillairet.Wittwer-Asymptoticdescriptionof2011}
and we refer the reader to this article for some technical details
which are omitted here.
\begin{thm}[weak-strong uniqueness]
\label{thm:uniqueness}\index{Uniqueness!weak solutions in the half-plane}\index{Weak solutions!in the half-plane!uniqueness}\index{Navier-Stokes equations!in the half-plane!uniqueness}Let
be $\bar{\bu}$ a weak solution that satisfies
\begin{align}
\bar{\bu} & \in D_{\sigma}^{1,2}(\Omega)\,, & y\bar{\bu} & \in L^{\infty}(\Omega)\,, & y^{2}\bnabla\bar{\bu} & \in L^{\infty}(\Omega)\,,\label{eq:bound-ubar}
\end{align}
and such that $\left\Vert y\bar{\bu}\right\Vert _{\infty}$ is small
enough. Then any weak solution $\bu$ with boundary value $\bu^{*}=\left.\bar{\bu}\right|_{\partial\Omega}$
that satisfies the energy inequality
\begin{equation}
\int_{\Omega}\bnabla\bu:\bnabla\bv\leq\int_{\Omega}\left(\bu\bcdot\bnabla\bv\right)\ba\,,\label{eq:energy-inequality}
\end{equation}
coincides with $\bar{\bu}$.\end{thm}
\begin{rem}
The $(\alpha,q)$-solutions found in \thmref{Navier-Stokes}
satisfy the requirement \eqref{bound-ubar} on $\bar{\bu}$.
\end{rem}
The remaining part of this section is devoted to the proof of this
theorem. To begin with, we prove that integration by parts with respect
to the solution $\bar{\bu}$ is permitted:
\begin{lem}[integration by parts]
\textup{\label{lem:int-by-parts}For any $\bar{\bu}$ that
satisfies \eqref{bound-ubar}, we have
\[
\int_{\Omega}\left(\bw\bcdot\bnabla\bar{\bu}\right)\bcdot\bu+\int_{\Omega}\left(\bw\bcdot\bnabla\bu\right)\bcdot\bar{\bu}=0\,,
\]
for all $\bu,\bw\in D_{\sigma}^{1,2}(\Omega)$ with $\bu/y\in L^{2}(\Omega)$
and $\bw/y\in L^{2}(\Omega)$. We note in particular, that if $\bu$
and $\bw$ are weak solutions, the hypothesis are satisfied.}\end{lem}
\begin{proof}
By using Hölder inequality, we have the bounds
\[
\left|\int_{\Omega}\left(\bw\bcdot\bnabla\bar{\bu}\right)\bcdot\bu\right|\leq\left\Vert \bw/y\right\Vert _{2}\left\Vert y^{2}\bnabla\bar{\bu}\right\Vert _{\infty}\left\Vert \bu/y\right\Vert _{2}\,,
\]
and
\[
\left|\int_{\Omega}\left(\bw\bcdot\bnabla\bu\right)\bcdot\bar{\bu}\right|\leq\left\Vert \bw/y\right\Vert _{2}\left\Vert \bnabla\bu\right\Vert _{2}\left\Vert y\bar{\bu}\right\Vert _{\infty}\,.
\]
Then, the result follows by an integration by parts, where $\bar{\bu}$
is approximated by compactly supported functions \citep[see for example][Proposition 21]{Hillairet.Wittwer-Asymptoticdescriptionof2011}.

Finally, if $\bu=\ba+\bv$ is a weak solution, we have by hypothesis
$\ba/y\in L^{2}(\Omega)$ and by Hardy inequality $\bv/y\in L^{2}(\Omega)$,
since $\bv\in D_{0,\sigma}^{1}$.
\end{proof}
Next we prove some results on the extension of allowed test functions
in the definition of weak solutions:
\begin{lem}
\label{lem:vbar-in-u}If $\bu$ is a weak solution, then
\[
\int_{\Omega}\bnabla\bu:\bnabla\bar{\bv}+\int_{\Omega}\left(\bu\bcdot\bnabla\bu\right)\bcdot\bar{\bv}=0\,,
\]
for any $\bar{\bv}\in D_{0,\sigma}^{1}(\Omega)$ such that $y\bar{\bv}\in L^{\infty}(\Omega)$.\end{lem}
\begin{proof}
We have
\begin{align*}
\left|\int_{\Omega}\bnabla\bu:\bnabla\bphi+\int_{\Omega}\left(\bu\bcdot\bnabla\bu\right)\bcdot\bphi\right| & \leq\left|\int_{\Omega}\bnabla\bu:\bnabla\bphi\right|+\left|\int_{\Omega}\left(\ba\bcdot\bnabla\bu\right)\bcdot\bphi\right|+\left|\int_{\Omega}\left(\bv\bcdot\bnabla\bu\right)\bcdot\bphi\right|\\
 & \begin{aligned}\leq\left\Vert \bnabla\bu\right\Vert _{2}\left\Vert \bnabla\bphi\right\Vert _{2} & +\left\Vert y\ba\right\Vert _{\infty}\left\Vert \bnabla\bu\right\Vert _{2}\left\Vert \bphi/y\right\Vert _{2}\\
 & +\left\Vert \bv/y\right\Vert _{2}\left\Vert \bnabla\bu\right\Vert _{2}\left\Vert y\bphi\right\Vert _{\infty}
\end{aligned}
\\
 & \leq\left\Vert \bnabla\bu\right\Vert _{2}\left(\left\Vert \bnabla\bphi\right\Vert _{2}+2\left\Vert y\ba\right\Vert _{\infty}\left\Vert \bnabla\bphi\right\Vert _{2}+2\left\Vert \bnabla\bv\right\Vert _{2}\left\Vert y\bphi\right\Vert _{\infty}\right)\,,
\end{align*}
so that the expression under consideration defines a linear form in
$\varphi$. Since we can approximate $\bar{\bv}$ by $\left(\bar{\bv}_{n}\right)_{n\in\mathbb{N}}\in C_{0,\sigma}^{\infty}(\Omega)$,
such that $\left\Vert \bnabla\bar{\bv}-\bnabla\bar{\bv}_{n}\right\Vert _{2}+\left\Vert y\bar{\bv}-y\bar{\bv}_{n}\right\Vert _{\infty}\to0$
as $n\to\infty$, this proves the lemma.\end{proof}
\begin{lem}
\label{lem:v-in-ubar}If $\bar{\bu}$ is a weak solution
such that $y\bar{\bu}\in L^{\infty}(\Omega)$, then
\[
\int_{\Omega}\bnabla\bar{\bu}:\bnabla\bv+\int_{\Omega}\left(\bar{\bu}\bcdot\bnabla\bar{\bu}\right)\bcdot\bv=0\,,
\]
for any $\bv\in D_{0,\sigma}^{1}(\Omega)$.\end{lem}
\begin{proof}
We have
\begin{align*}
\left|\int_{\Omega}\bnabla\bar{\bu}:\bnabla\varphi+\int_{\Omega}\left(\bar{\bu}\bcdot\bnabla\bar{\bu}\right)\bcdot\bphi\right| & \leq\left|\int_{\Omega}\bnabla\bar{\bu}:\bnabla\bphi\right|+\left|\int_{\Omega}\left(\bar{\bu}\bcdot\bnabla\bar{\bu}\right)\bcdot\bphi\right|\\
 & \leq\left\Vert \bnabla\bar{\bu}\right\Vert _{2}\left\Vert \bnabla\bphi\right\Vert _{2}+\left\Vert y\bar{\bu}\right\Vert _{\infty}\left\Vert \bnabla\bar{\bu}\right\Vert _{2}\left\Vert \bphi/y\right\Vert _{2}\\
 & \le\left\Vert \bnabla\bar{\bu}\right\Vert _{2}\left(1+2\left\Vert y\bar{\bu}\right\Vert _{\infty}\right)\left\Vert \bnabla\bphi\right\Vert _{2}\,,
\end{align*}
and since the form is linear in $\bphi$, the lemma is proved.
\end{proof}
We now prove that the weak solution $\bar{\bu}$ satisfies an energy
equality:
\begin{lem}
Any weak solution $\bar{\bu}$ which satisfies \eqref{bound-ubar}
verifies the energy equality
\begin{equation}
\int_{\Omega}\bnabla\bar{\bu}:\bnabla\bar{\bv}=\int_{\Omega}\left(\bar{\bu}\bcdot\bnabla\bar{\bv}\right)\ba\,.\label{eq:energy-equality}
\end{equation}
\end{lem}
\begin{proof}
By \lemref{v-in-ubar}, we have
\[
\int_{\Omega}\bnabla\bar{\bu}:\bnabla\bar{\bv}+\int_{\Omega}\left(\bar{\bu}\bcdot\bnabla\bar{\bu}\right)\bcdot\bar{\bv}=0\,,
\]
and by \lemref{int-by-parts},
\begin{align*}
\int_{\Omega}\left(\bar{\bu}\bcdot\bnabla\bar{\bv}\right)\bcdot\bar{\bv} & =0\,, & \int_{\Omega}\left(\bar{\bu}\bcdot\bnabla\ba\right)\bcdot\bar{\bv}+\int_{\Omega}\left(\bar{\bu}\bcdot\bnabla\bar{\bv}\right)\bcdot\ba & =0\,,
\end{align*}
so we obtain the energy equality
\[
\int_{\Omega}\bnabla\bar{\bu}:\bnabla\bar{\bv}=-\int_{\Omega}\left(\bar{\bu}\bcdot\bnabla\ba\right)\bcdot\bar{\bv}=\int_{\Omega}\left(\bar{\bu}\bcdot\bnabla\bar{\bv}\right)\bcdot\ba\,.
\]

\end{proof}
We now have the necessary tools in order to prove the main theorem
of this section:
\begin{proof}[Proof of \thmref{uniqueness}]
Let $\bu$ and $\bar{\bu}$ be two weak solutions with the same boundary
conditions, so $\bd=\bu-\bar{\bu}\in D_{0,\sigma}^{1,2}(\Omega)$,
see for example \citet[Theorem II.7.7]{Galdi-IntroductiontoMathematical2011}.
Then, by using the scalar product on $D_{0,\sigma}^{1,2}(\Omega)$,
we have
\begin{align*}
\left\Vert \bnabla\bd\right\Vert _{2}^{2} & =\int_{\Omega}\left(\bnabla\bu-\bnabla\bar{\bu}\right):\left(\bnabla\bv-\bnabla\bar{\bv}\right)\\
 & =\int_{\Omega}\bnabla\bu:\bnabla\bv+\int_{\Omega}\bnabla\bar{\bu}:\bnabla\bar{\bv}-\int_{\Omega}\bnabla\bu:\bnabla\bar{\bv}-\int_{\Omega}\bnabla\bar{\bu}:\bnabla\bv\,.
\end{align*}
By using \lemref{vbar-in-u, v-in-ubar}, the energy
equality \eqref{energy-equality} and the energy inequality
\eqref{energy-inequality}, we have
\[
\left\Vert \bnabla\bd\right\Vert _{2}^{2}\leq\int_{\Omega}\left(\bu\bcdot\bnabla\bv\right)\bcdot\ba+\int_{\Omega}\left(\bar{\bu}\bcdot\bnabla\bar{\bv}\right)\bcdot\ba+\int_{\Omega}\left(\bu\bcdot\bnabla\bu\right)\bcdot\bar{\bv}+\int_{\Omega}\left(\bar{\bu}\bcdot\bnabla\bar{\bu}\right)\bcdot\bv\,.
\]
Since $\bu/y\in L^{2}(\Omega)$ and $\bar{\bu}/y\in L^{2}(\Omega)$
by using Hardy inequality and \lemref{int-by-parts} we
have
\begin{align*}
\int_{\Omega}\left(\bu\bcdot\bnabla\ba\right)\bcdot\ba & =0\,, & \int_{\Omega}\left(\bar{\bu}\bcdot\bnabla\ba\right)\bcdot\ba & =0\,,
\end{align*}
which allows us to rewrite the bound as
\begin{align*}
\left\Vert \bnabla\bd\right\Vert _{2}^{2} & \leq\int_{\Omega}\left(\bu\bcdot\bnabla\bu\right)\bcdot\ba+\int_{\Omega}\left(\bar{\bu}\bcdot\bnabla\bar{\bu}\right)\bcdot\ba+\int_{\Omega}\left(\bu\bcdot\bnabla\bu\right)\bcdot\bar{\bv}+\int_{\Omega}\left(\bar{\bu}\bcdot\bnabla\bar{\bu}\right)\bcdot\bv\\
 & \leq\int_{\Omega}\left(\bu\bcdot\bnabla\bu\right)\bcdot\bar{\bu}+\int_{\Omega}\left(\bar{\bu}\bcdot\bnabla\bar{\bu}\right)\bcdot\bu\,.
\end{align*}
By using \lemref{int-by-parts}, we integrate the second
term by parts,
\[
\left\Vert \bnabla\bd\right\Vert _{2}^{2}\leq\int_{\Omega}\left(\bu\bcdot\bnabla\bu\right)\bcdot\bar{\bu}-\int_{\Omega}\left(\bar{\bu}\bcdot\bnabla\bu\right)\bcdot\bar{\bu}=\int_{\Omega}\left(\bd\bcdot\bnabla\bu\right)\bcdot\bar{\bu}\,.
\]
Again by \lemref{int-by-parts}, we have
\[
\int_{\Omega}\left(\bd\bcdot\bnabla\bar{\bu}\right)\bcdot\bar{\bu}=0\,,
\]
so by Hardy inequality,
\[
\left\Vert \bnabla\bd\right\Vert _{2}^{2}\leq\left|\int_{\Omega}\left(\bd\bcdot\bnabla\bd\right)\bcdot\bar{\bu}\right|\leq\left\Vert \bd/y\right\Vert _{2}\left\Vert \bnabla\bd\right\Vert _{2}\left\Vert y\bar{\bu}\right\Vert _{\infty}\leq2\left\Vert y\bar{\bu}\right\Vert _{\infty}\left\Vert \bnabla\bd\right\Vert _{2}^{2}\,.
\]
Therefore, if $\left\Vert y\bar{\bu}\right\Vert _{\infty}$ is small
enough, we obtain that $\bd=\bzero$, \emph{i.e.} $\bu=\bar{\bu}$.
\end{proof}

\section{Numerical simulations\index{Numerical simulations!in the half-plane}}

In order to simulate this problem numerically, we truncate the domain
$\Omega$ to a ball of radius $R=10^{4}$, $\Omega_{R}=\Omega\cap B((0,1),R)$.
On the bottom boundary we take an antisymmetric perturbation of a
symmetric Jeffery-Hamel,
\begin{equation}
\left.\bu\right|_{\left[-R,R\right]\times\{1\}}=\bu_{\Phi}^{0}+\frac{\nu}{r}\sin(2\theta)\be_{r}\,,\label{eq:numerics-BC}
\end{equation}
and on the artificial boundary $\Gamma$, which is the upper half
circle of radius $R$, we take 
\[
\left.\bu\right|_{\Gamma}=\bu_{\Phi}^{0}\,.
\]
In \figref{plot}, we represent the velocity field $\bu$
multiplied by $r$ in order to see the behavior at large distances.
For $\nu=0$, this corresponds to the Jeffery-Hamel solutions $\bu_{\Phi}^{0}$
which are scale-invariant. For negative fluxes $\Phi<0$, small perturbations
have almost no effect on the behavior at large distances, so the asymptotic
term is probably given by the Jeffery-Hamel solution $\bu_{\Phi}^{0}$.
Conversely for $\Phi>0$, even a small perturbation drastically change
the behavior of the solution at large distances by somehow rotating
the region where the magnitude of the velocity is large. In the case,
the asymptotic behavior is very likely not given by the Jeffery-Hamel
solution $\bu_{\Phi}^{0}$. This conclusion can also be seen in \figref{profile},
where we plot the velocity field $\bu$ in polar coordinates multiplied
by $r$ on the half-circle $\bigl\{10^{3}\left(\sin\theta,\cos\theta\right),\theta\in\bigl[\frac{-\pi}{2};\frac{\pi}{2}\bigr]\bigr\}$
in terms of $\nu$.

\begin{figure}[p]
\includefigure{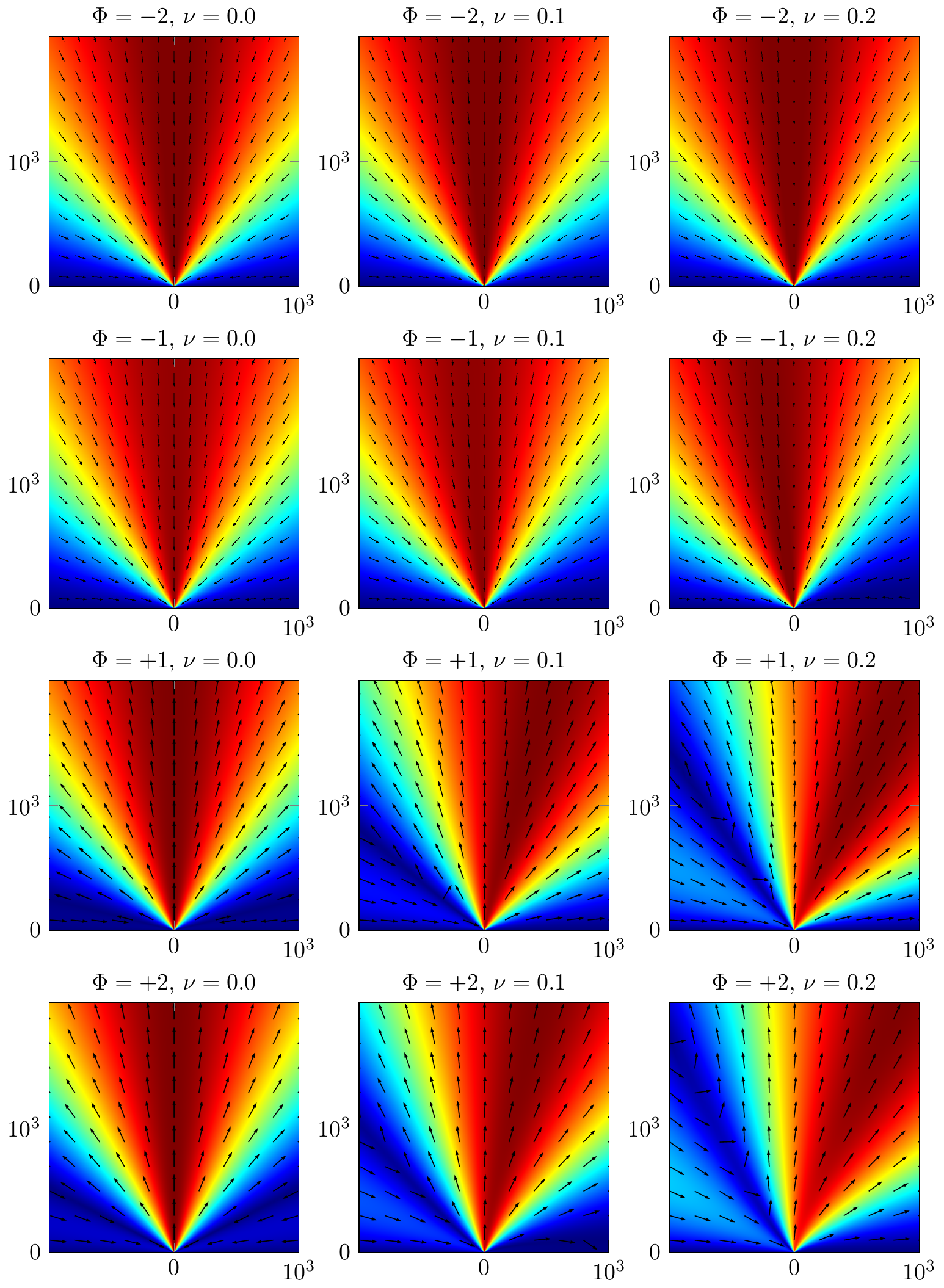}

\caption{\label{fig:plot}Numerical results for the velocity $\bu$
multiplied by $r$, for the boundary condition \eqref{numerics-BC}
in a domain of size $R=10^{4}$, for various values of the flux $\Phi$
and of the perturbation $\nu$.}
\end{figure}
\begin{figure}[p]
\includefigure{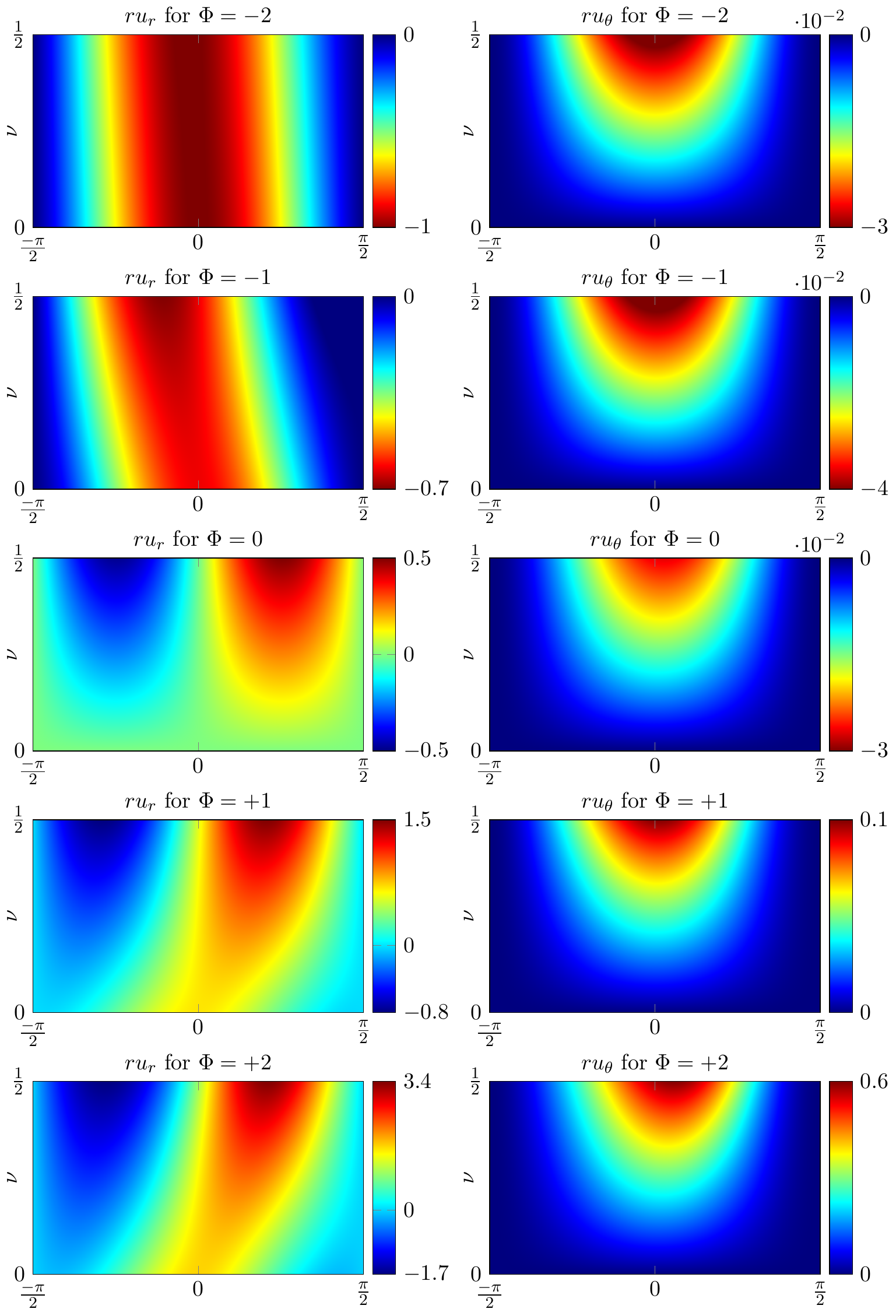}

\caption{\label{fig:profile}Profiles of $ru_{r}$ and $ru_{\theta}$
on the half-circle of radius $10^{3}$ in term of $\theta\in[-\pi/2,\pi/2]$
and $\nu\in[0,1]$.}
\end{figure}
\newpage{}

\appendix

\section{Jeffery-Hamel solutions with small flux\label{sec:JH-phi}}

A Jeffery-Hamel solution in the upper half-plane\index{Jeffery-Hamel solutions!existence},
\[
\Bigl\{\left(r\sin\theta,r\cos\theta\right),\, r>0\text{ and }\theta\in\big(\tfrac{-\pi}{2},\tfrac{\pi}{2}\big)\Bigr\}\,,
\]
is a radial solution of the Navier-Stokes equations with zero velocity
on the boundary and whose velocity norm is $f(\theta)/r$. More explicitly
$f$ has to satisfy the boundary value problem \eqref{JH&BC}. Here
we prove an existence theorem:
\begin{thm}
\label{thm:JH-phi}For every small enough value of the flux
$\phi$, the Jeffery-Hamel equation 
\begin{align}
f^{\prime\prime}+f^{2}+4f & =2C\,, & f\big(\tfrac{\pm\pi}{2}\big) & =0\,,\label{eq:JH&BC}
\end{align}
admits a symmetric solution,
\[
f_{\phi}^{0}(\theta)=\frac{2\phi}{\pi}\cos^{2}(\theta)+O(\phi^{3/2})\,,
\]
and in addition if $\phi<0$, two quasiantisymmetric solutions,
\[
f_{\phi}^{\pm1}(\theta)=\pm\sqrt{\frac{-48\phi}{\pi}}\sin(2\theta)+O(\phi)\,.
\]
\end{thm}
\begin{proof}
First of all the solution of the linear equation
\[
f^{\prime\prime}+4f=g\,,
\]
is given by
\[
f(\theta)=A\sin(2\theta)+B\cos(2\theta)-L[g](\theta)\,,
\]
where
\[
L[g](\theta)=\frac{1}{2}\cos(2\theta)\int_{-\pi/2}^{\theta}g(s)\sin(2s)\,\rd s-\frac{1}{2}\sin(2\theta)\int_{-\pi/2}^{\theta}g(s)\cos(2s)\,\rd s\,.
\]
Therefore the Jeffery-Hamel equation and the boundary condition \eqref{JH&BC},
can be rewritten as
\begin{align}
f(\theta) & =A\sin(2\theta)+C\cos^{2}(\theta)+L[f^{2}](\theta)\,, & \int_{-\pi/2}^{+\pi/2}f^{2}(\theta)\sin(2\theta)\,\rd\theta & =0\,.\label{eq:JH&BC-2}
\end{align}
By defining
\begin{align*}
f_{0}(\theta) & =A\sin(2\theta)\,, & f_{1}(\theta) & =C\cos^{2}(\theta)-\frac{A^{2}}{3}\cos^{4}(\theta)\,, & f & =f_{0}+f_{1}+\bar{f}\,,
\end{align*}
the flux condition
\[
\phi=\int_{-\pi/2}^{+\pi/2}f(\theta)\,\rd\theta
\]
directly gives the definition of $C$ in term of the flux,
\[
C=\frac{2\phi}{\pi}+\frac{A^{2}}{4}-\frac{2}{\pi}\int_{-\pi/2}^{+\pi/2}\bar{f}(\theta)\,\rd\theta\,,
\]
and the two equations \eqref{JH&BC-2} can be rewritten by substitution
as:
\begin{align}
\bar{f} & =L\left[\left(2f_{0}+f_{1}+\bar{f}\right)\left(f_{1}+\bar{f}\right)\right]\,,\label{eq:JH-F-FP}\\
A\left(\phi+\frac{\pi}{48}A^{2}\right) & =\int_{-\pi/2}^{+\pi/2}\left(A-2f_{0}(\theta)-2f_{1}(\theta)-\bar{f}(\theta)\right)\bar{f}(\theta)\,\rd\theta\,.\label{eq:JH-A}
\end{align}
In order to find a fixed point of theses equations, we first solve
the left-hand-side of the second equation
\[
A_{0}\left(\phi+\frac{\pi}{48}A_{0}^{2}\right)=0
\]
for $A_{0}$. This equation admits the solution $A_{0}=0$ and in
addition if $\phi<0$ the two solutions
\[
A_{0}=\pm\sqrt{\frac{-48\phi}{\pi}}\,.
\]
Given one of these three solutions, we define
\[
A=A_{0}+\bar{A}\,,
\]
and \eqref{JH-A} becomes:
\begin{equation}
\bar{A}=\frac{48}{48\phi+3\pi A_{0}^{2}}\left[\int_{-\pi/2}^{+\pi/2}\left(A-2f_{0}(\theta)-2f_{1}(\theta)-\bar{f}(\theta)\right)\bar{f}(\theta)\,\rd\theta-\frac{\pi}{48}\left(3A_{0}+\bar{A}\right)\bar{A}^{2}\right]\,.\label{eq:JH-A-FP}
\end{equation}
It is easily verified that the maps defined by \eqref{JH-A-FP} and
\eqref{JH-F-FP} map the ball 
\[
B_{\phi}=\left\{ \left(\bar{A},\bar{f}\right)\in\mathbb{R}\times C^{0}\big(\big[-\tfrac{\pi}{2},\tfrac{\pi}{2}\big]\big)\colon\left|\bar{A}\right|\leq50\phi\text{ and }\left|\bar{f}\right|\leq600\phi^{3/2}\right\} 
\]
into itself, provided $\phi$ is small enough. Moreover, since the
maps \eqref{JH-A-FP} and \eqref{JH-F-FP} are multilinear affine
maps of $\bar{A}$ and $\bar{f}$, they are contractions from $B_{\phi}$
into itself, for $\phi$ small enough. The first order terms which
we explicitly computed above prove the claimed leading terms of the
symmetric and quasiantisymmetric solutions.
\end{proof}
\newpage{}

\bibliographystyle{merlin-dot}
\phantomsection\addcontentsline{toc}{section}{\refname}\bibliography{paper}

\end{document}